\documentclass[11pt,letterpaper]{amsart}

\usepackage{enumitem} 
\usepackage{amsmath,amssymb,amsthm,amsxtra,mathrsfs}
\usepackage{graphicx}
\usepackage{dsfont}
\usepackage{cancel,ulem}
\usepackage{color}
\usepackage[all,cmtip]{xy}
\usepackage{hyperref}
\usepackage{enumitem}
\usepackage{color}
\usepackage{bm}

\usepackage{subcaption}
\usepackage{tikz}
\usetikzlibrary{arrows}
\usepackage{accents}
\usepackage{bbm}

\def\r{{\rangle}}
\def\l{{\langle}}
\def\S{{\mathbb S}}

\def\R{{\mathbb R}}

\newtheorem{theorem}{Theorem}
\newtheorem{lemma}[theorem]{Lemma}
\newtheorem{corollary}[theorem]{Corollary}

\theoremstyle{definition}
\newtheorem{definition}[theorem]{Definition}

\theoremstyle{remark}
\newtheorem{remark}[theorem]{Remark}

\numberwithin{equation}{section}
\numberwithin{theorem}{section}
\numberwithin{problem}{section}

\begin{document}

\begin{abstract}
In this article we identify a sharp ill-posedness/well-posedness threshold for kinetic wave equations (KWE) derived from quasilinear Schr\"{o}dinger models. 

We show well-posedness using a collisional averaging estimate proved in our earlier work \cite{AmLe}. Ill-posedness manifests as instantaneous loss of smoothness for well-chosen initial data.

We also prove that both the gain-only and full equation share the same well-posedness threhold, thus legitimizing a gain-only approach to solving 4-wave kinetic equations.
\end{abstract}

\title{On the ill-posedness of kinetic wave equations}

\author[I. Ampatzoglou and T. L\'eger]{Ioakeim Ampatzoglou and Tristan L\'eger}

\address{Baruch College, The City University of New York, Newman Vertical Campus, 55 Lexington Ave New York, NY, 10010, USA
}
\email{ioakeim.ampatzoglou@baruch.cuny.edu}%

\address{Princeton University,  Mathematics  Department,  Fine Hall, Washington Road,  Princeton,  NJ 08544-1000,  USA}
\email{tleger@princeton.edu}

\maketitle

\tableofcontents

\section{Introduction}
\subsection{Context}
Kinetic equations are widely used to approximate complicated systems of many interacting agents (particles, waves,...) that cannot directly be simulated numerically due to their immense size. 
The field of derivation of such equations for weakly nonlinearly interacting waves has seen impressive advances recently: after a series of works on the rigorous derivation of the 4-wave kinetic wave equation (KWE) from the cubic nonlinear Schr\'odinger equation (NLS) see e.g. \cite{BGHS,CG1,CG2,DH1}, the state of the art result was obtained by Z. Hani and Y. Deng \cite{DH2}: they derived the equation up to the kinetic time and later showed validity of the (KWE) for as long as it is well-posed \cite{DH4}. Recently, the same authors along with X. Ma \cite{DHM} extended their program to interacting particles and provided a long time derivation of the Boltzmann equation, again subject to the condition of well-posedness. Besides the cubic (NLS), other  models such as KdV-type \cite{ST,Ma} or stochastic KP equations \cite{Faou} have also been considered  to derive 3-wave kinetic equations. Spatially inhomogeneous turbulence has also been studied and effective transport equations have been derived in  \cite{AmCoGer, HaRoStTr, HaShZh}.

However, despite the promising results in the field of derivation, the understanding of the kinetic equations themselves  remains poor. Yet this is a crucial step in the derivation program since all such results are conditional on existence of a smooth solution to the corresponding kinetic equation.

In this paper we seek to make some progress in this direction, and address local well-posedness for a family of kinetic wave equations derived from quasilinear Schr\"{o}dinger systems. Regarding the choice of the model, we are motivated by the fact that derivation of kinetic equations from quasilinear systems seems challenging. The results available are less convincing than for semilinear equations (meaning their validity is proved on much smaller time scales). A possible explanation is that kinetic theory does not apply as well to quasilinear equations, which would manifest as ill-posedness of formally derived corresponding kinetic approximations. The choice of the equations here allows us to correlate precisely the amount of derivative loss (measured by the parameter $\beta$, see \eqref{MMT}) with the local well-posedness of the kinetic wave equation.   

Our analysis identifies a sharp threshold of derivative loss for well-posedness of the equation to hold. That is, if the original equation loses more derivatives than that threshold, then its kinetic equation is ill-posed locally in time. Conversely if it loses less, then the kinetic equation is well-posed. Note that the original model of interacting waves appears to remain well-posed while the kinetic equation becomes ill-posed (see Remark \ref{MMT-well-posed} below). We provide possible explanations for the numerical value of the threshold, arguing both on mathematical and physical grounds.

A second insight provided by our analysis is that the threshold is the same for the full equation and its gain-only counterpart. Interestingly however, ill-posedness is achieved in different ways for both equations. Indeed the full equation has some cancellations that are not quite strong enough to restore well-posedness, but do dampen low frequencies.  

\subsection{The model}
\subsubsection{Definition}
Let $0\leq \beta \leq 1.$ The object of this article is the kinetic theory for systems of waves governed by the quasilinear equations
\begin{align} \label{MMT}
    i \partial_t u + \Delta u = \vert \nabla \vert^{\beta} \big(\vert \nabla \vert^\beta u \big \vert \vert \nabla \vert^{\beta} u \big \vert^2 \big)
\end{align}
for $u:(-T,T) \times \mathbb{T}^3 \rightarrow \mathbb{C},$ $T>0.$

The operator $\vert \nabla \vert^{\beta}$ is defined as 
\begin{align*}
\vert \nabla \vert^{\beta} e^{ik \cdot x} = \vert k \vert^{\beta} e^{ik \cdot x} .
\end{align*}
Written in $k$ space, the corresponding Hamiltonian is 
\begin{align} \label{Hamiltonian}
H = \sum_{k_1 \in \mathbb{Z}^3} \vert k_1 \vert^2 a_{k_1} a_{k_1}^{*} + \frac{1}{2} \sum_{k_1,k_2,k_3,k_4 \in \mathbb{Z}^3 } \big(\vert k_1 \vert \vert k_2 \vert \vert k_3 \vert \vert k_4 \vert \big)^\beta a_{k_1} a_{k_2} a_{k_3}^{*} a_{k_4}^{*}.
\end{align}
As mentioned above, the parameter $\beta$ quantifies the ``degree of quasilinearity" of the equation. The restriction $\beta \leq 1$ guarantees that the equation is indeed quasilinear.

From \eqref{Hamiltonian} it is then standard to derive (formally) the corresponding homogeneous kinetic wave equation:
\begin{align} \label{DKWE}
\begin{split}
    \partial_t f &= \mathcal{C}[f], \\ 
    \mathcal{C}[f] & := \vert k_1 \vert^{2\beta} \int_{\mathbb{R}^{9}} \big(\vert k_2 \vert \vert k_3 \vert \vert k_4 \vert \big)^{2\beta} f_1f_2f_3 f_4 \big(\frac{1}{f_1} + \frac{1}{f_2} - \frac{1}{f_3} - \frac{1}{f_4} \big)  \delta(\Sigma)   \delta(\Omega) \, dk_2 dk_3 dk_4, 
    \\
    \Sigma &:= k_1+k_2-k_3-k_4 \\
    \Omega &:= \vert k_1 \vert^2 + \vert k_2 \vert^2 + \vert k_3 \vert^2 + \vert k_4 \vert^2 .
\end{split}
\end{align}
We refer to classical physics textbooks on the topic such as \cite{Na} for the details of the procedure.

Another case of interest if the so-called inhomogeneous equation. It is derived from \eqref{MMT} set on $\mathbb{R}^3$ and reads $\partial_t f + v \cdot \nabla_x f = \mathcal{C}[f].$ 

\subsubsection{Prior works}
In one spatial dimension \eqref{DKWE} corresponds to the celebrated MMT equations \cite{MMT} with Laplacian dispersion relation. This model is widely used to test weak turbulence theory. As such it has been investigated extensively both numerically and theoretically. Closest to the results of the present paper, we mention the work \cite{GLZ} which proves local well-posedness of MMT ($1d$) in the case $-\frac{1}{2} \leq \beta \leq 0$ for general dispersion relation. We note that we assume on the contrary that $0 \leq \beta \leq 1$ since we are interested in the effect of quasilinearity on the well-posedness of the kinetic wave equation. In kinetic theory parlance, we treat hard potentials, while \cite{GLZ} deals with soft potentials. 

The case $\beta =0$ corresponds to one of the most canonical 4-wave kinetic equations. It is derived from the cubic Schr\"{o}dinger equation. The paper \cite{EV} by M. Escobedo and J.J.L. Velázquez constitutes an exhaustive study (local and global existence, asymptotic behavior) of weak solutions in the isotropic case. In \cite{CDG}, C. Collot, P. Germain and H. Dietert study in great detail the stability of Kolmogorov-Zakharov solutions for the same model. See also the work of A. Menegaki \cite{me23} for stability near equilibrium and the recent work of M. Escobedo and A. Menegaki \cite{EsMe24} for the instability of singular equilibria.     In the non-isotropic case, \cite{GeIoTr} provides local well-posedness for strong solutions in weighed $L^2$ spaces (still for $\beta=0$). 

We should stress that in the space inhomogeneous case, the (KWE) exhibits better long time asymptotics due to the dispersion introduced by the transport. Global well posedness in the mild sense has been first shown in \cite{Am} for exponentially decaying data and later for polynomially decaying data in space and frequency in \cite{AmMiPaTa24}, where the corresponding kinetic wave hierarchy was also studied. Recently in \cite{AmLe}, we exploited finer dispersion properties of the collisional operator, and constructed global in time strong solutions polynomially close to vacuum in frequency only, which propagate moments and exhibit scattering behavior.

\subsection{Results}
\subsubsection{Statement}
We formulate our well-posedness result in polynomially weighted spaces. Indeed the relevant solutions in the physical theory, namely Rayleigh-Jeans distributions (thermodynamic equilibrium) and Kolmogorov-Zakharov spectra (out-of-equilibrium) are of this type. It could also be interesting to consider exponentially weighted spaces in relation to numerical simulations. However we expect the method of proof in this case to be rather different than in the present paper.

A simplified version of our main result is
\begin{theorem}
Assume that $0 \leq \beta \leq \frac{1}{4}.$ Then \eqref{DKWE} is locally well-posed in weighted $L^\infty$ spaces. 

If $\beta>\frac{1}{4}$ then both \eqref{DKWE} and its gain-only counterpart are ill-posed in weighted $L^\infty$ spaces, in the sense that strong solutions cannot be constructed (regardless of the size of the initial datum).
\end{theorem}
Although we stated and proved the theorem for the homogeneous equation, our proof applies to the inhomogeneous case as well.

The well-posedness statement is obtained as a consequence of the averaging effect proved in our earlier work \cite{AmLe}. It allows us to cover the entire range, and therefore confirms that it is the appropriate tool to use in the setting of kinetic equations of systems of four wave interaction. In particular global existence and scattering results can be proved in this regime for the inhomogeneous version of \eqref{DKWE} following the method of \cite{AmLe}.

Our strategy for ill-posedness is to exhibit an initial datum such that the second Picard iterate is unbounded. We note that its size can be scaled, hence ill-posedness is true also in the small data regime. In the case of gain-only equation this is simply achieved by picking $\langle k_1 \rangle^{-M}$ as the initial datum. In the case of the full equation though, this strategy does not work due to cancellations in the collision operator (see Remark \ref{beta34} below). However adding oscillations destroys this effect.

\subsubsection{Physical interpretation}
We make several remarks about our results.
\begin{remark} \label{MMT-well-posed}
    The thresholds for well-posedness of \eqref{MMT} and \eqref{DKWE} do not appear to coincide. Indeed works of C. Kenig. G. Ponce and L. Vega show that Schr\"{o}dinger equations with polynomial nonlinearities involving $\nabla_x u, \nabla_x \bar{u},u,  \bar{u}$ are locally well-posed in $H^s,$ see \cite{KPV1}, \cite{KPV2}. For $\beta=\frac{1}{2},$ the nonlinearity in \eqref{MMT} is formally of this type, thus local well-posedness is strongly expected.
\end{remark}

\begin{remark}
$\beta=\frac{1}{4}$ corresponds to the largest exponent such that the estimate 
\begin{align} \label{betaSobolev}
    \Vert \vert \nabla \vert^\beta u \Vert_{L^4} \lesssim \Vert u \Vert_{\dot{H}^1}
\end{align}
holds by Sobolev embedding. This is consistent with the general principle that in weak turbulence, the existence of energy cascade is based on (approximate) conservation of kinetic energy (by analogy with hydrodynamic turbulence). For this heuristic to hold, the potential energy must be controlled by the kinetic part, which \eqref{betaSobolev} encodes.
\end{remark}

\begin{remark} \label{beta34}
The threshold $\beta = \frac{3}{4}$ is also natural. It is associated with boundedness of the second Picard iterate for the initial datum $\langle k_1 \rangle^{-M}.$ 

More precisely, for $\beta > \frac{3}{4}$ we have $\mathcal{C}[\langle k_1 \rangle^{-M}] \notin \langle k_1 \rangle^{-M} L^{\infty}_{k_1}.$ On the contrary $\mathcal{C}[\langle k_1 \rangle^{-M}] \in \langle k_1 \rangle^{-M} L^{\infty}_{k_1}$ for $0 \leq \beta \leq \frac{3}{4}$. This is in contrast with the gain only equation where the corresponding threshold is $\beta = \frac{1}{4}.$ The difference between the two is due to cancellations in the collision operator of the full equation. We rely on these cancellations in the proof of the main result, see subsection \ref{I4}.
\end{remark}

\begin{remark}
These two thresholds also correspond to the infinite/finite capacity dichotomy for the system in the direct energy and inverse waveaction cascades respectively.

Indeed, following \cite{Na} section 9.2, it is possible to derive the corresponding spectra on physical grounds: we take the ansatz $f(k_1) \sim k_1^{\nu},$ (in the isotropic case here, so $k_1$ denotes $\vert k_1 \vert$). Recall that the energy flux $\epsilon(k_1)$ satisfies $\partial_t E^{(1D)} + \partial_{k_1} \epsilon = 0,$ where $E^{(1D)}(k_1) := k_1^4 f_{k_1} $ denotes the energy spectrum.  

As a consequence we have 
\begin{align*}
    \epsilon(k_1) &\sim {k_1}^{2\beta} \int_{\mathbb{S}^{6}} \delta(\Sigma) d\sigma_2 d\sigma_3 d\sigma_4 \int_0^{k_1} q^2  \int_{[0;\infty)^3}  (k_2 k_3 k_4)^{2\beta} f_2 f_3 f_4 \delta(\Omega) (k_2 k_3 k_4 q)^{2}  dk_2 dk_3 dk_4 dq \\ 
    &+ \lbrace \textrm{similar terms} \rbrace \\
    & \sim \underbrace{k_1^{8 \beta}}_{\textrm{interaction}} \cdot \underbrace{k_1^{-3}}_{\delta(\Sigma)} \cdot \underbrace{k_1^{3 \nu}}_{f_2 f_3 f_4} \cdot \underbrace{k_1^{2}}_{\textrm{energy}} \cdot \underbrace{k_1^{-2}}_{\delta(\Omega)} \cdot\underbrace{k_1^{12}}_{\textrm{integration}}.
\end{align*}
Constant flux requires $\epsilon \sim k_1^0,$ hence $ \nu = -3-\frac{8}{3} \beta.$ The energy spectrum is then $E^{(1D)}(k_1) \sim k_1^{1-\frac{8}{3}\beta}.$ This is integrable if $\beta<\frac{3}{4}$ (finite capacity) and infinite capacity otherwise.

For the inverse waveaction cascade $E^{(1D)}(k_1) = k_1^2 f_{k_1},$ and by a similar reasoning we find that the flux scales like $\epsilon(k_1) \sim k_1^{-\frac{7}{3}-\frac{8 \beta}{3}}.$ Therefore the threshold between finite and infinite capacity is $\beta = \frac{1}{4}.$ 

\end{remark}

\begin{remark}
   At the mathematical level, the fact that both the gain-only and full equation have the same threshold legitimizes a gain-only approach to solving \eqref{DKWE}. This can be done by exploiting the monotonicity of the collision operator with the Kaniel-Shinbrot scheme for example. We stress that the situation should be different for the Boltzmann equation. Indeed the work of X. Chen and J. Holmer \cite{CH} shows that the loss and gain terms are not bounded on the same spaces.
\end{remark}

\subsection{Parametrization}
The collision operator can be written in gain-loss form as follows:
\begin{equation}\label{decomposition in gain and loss standard}
    \mathcal{C}[f]=\mathcal{G}[f]-\mathcal{L}[f],
\end{equation}
where
\begin{align} \label{gain def}
    \mathcal{G}[f] = \vert k_1 \vert^{2\beta} \int_{\mathbb{R}^{9}} \big(\vert k_2 \vert \vert k_3 \vert \vert k_4 \vert \big)^{2\beta} \big(f_2f_3f_4 + f_1 f_3f_4 \big) \delta(\Sigma)   \delta(\Omega) \, dk_2 dk_3 dk_4,   
\end{align}
and 
\begin{align} \label{loss def}
    \mathcal{L}[f] = \vert k_1 \vert^{2\beta} \int_{\mathbb{R}^{9}} \big(\vert k_2 \vert \vert k_3 \vert \vert k_4 \vert \big)^{2\beta} \big( f_1 f_2f_4 +  f_1 f_2 f_3 \big) \delta(\Sigma)   \delta(\Omega) \, dk_2 dk_3 dk_4. 
\end{align}
For our well posedness result, we follow ideas from \cite{AmLe} and further decompose the gain and loss operators exploiting the symmetries of the equation.

Namely, introducing the cut-off $b:=\mathds{1}_{(0,+\infty)}$ and using the symmetry of the product $f_2f_3$ we can write
\begin{equation}\label{further decomposition of gain}
  \mathcal{G}[f] =\mathcal{G}_1[f,f,f] + \mathcal{G}_2[f,f,f],  
\end{equation}
where we denote
\begin{align}
\mathcal{G}_1[f,g,h]&:=  2\int_{\R^9}\big(\vert k_1 \vert \vert k_2 \vert \vert k_3 \vert \vert k_4 \vert\big)^{2\beta}\delta(\Sigma)\delta(\Omega)f_2g_3h_4 b\left(\left(k_1-k_2\right)\cdot\left(k_3-k_4\right)\right)\,dk_2\,dk_3\,dk_4\label{G_1},  \\
\mathcal{G}_2[f,g,h]&:=  2\int_{\R^9}\big(\vert k_1 \vert \vert k_2 \vert \vert k_3 \vert \vert k_4 \vert \big)^{2\beta}\delta(\Sigma)\delta(\Omega)f_1g_3h_4 b\left(\left(k_1-k_2\right)\cdot\left(k_3-k_4\right)\right)\,dk_2\,dk_3\,dk_4\label{G_2}.
\end{align}
Similarly, using the symmetry of the sum $f_2+f_3$ instead, we can write
\begin{equation}\label{loss decomposed}
   \mathcal{L}[f]=\mathcal{L}_1[f,f,f]+\mathcal{L}_2[f,f,f], 
\end{equation}
where
\begin{align}
\mathcal{L}_1[f,g,h]&:=  2\int_{\R^9}\big(\vert k_1 \vert \vert k_2 \vert \vert k_3 \vert \vert k_4 \vert  \big)^{2\beta}\delta(\Sigma)\delta(\Omega)f_1g_2h_4 b\left(\left(k_1-k_2\right)\cdot\left(k_3-k_4\right)\right)\,dk_2\,dk_3\,dk_4\label{L_1},  \\
\mathcal{L}_2[f,g,h]&:=  2\int_{\R^9}\big(\vert k_1 \vert \vert k_2 \vert \vert k_3 \vert \vert k_4 \vert \big)^{2\beta}\delta(\Sigma)\delta(\Omega)f_1g_2h_3 b\left(\left(k_1-k_2\right)\cdot\left(k_3-k_4\right)\right) \,dk_2\,dk_3\,dk_4\label{L_2}.
\end{align}
After parametrization of the resonant manifold (see \cite{AmLe} for more details) we can express the operators $\mathcal{G}_1,\mathcal{G}_2,\mathcal{L}_1,\mathcal{L}_2$  as follows
\begin{align}
\mathcal{G}_1[f,g,h]&=\frac{1}{4}\int_{\R^3\times\S^2}\big(\vert k_1 \vert \vert k_2 \vert \vert k_1^{*} \vert \vert k_2^{*} \vert \big)^{2\beta}|k_1-k_2|f(k_2)g(k_1^*)h(k^*_2) b\left(\left(k_1-k_2\right)\cdot\sigma\right) \,d\sigma\,dk_2    \label{G_1 parametrized}\\
\mathcal{G}_2[f,g,h]&=\frac{1}{4}f(k_1)\int_{\R^3\times\S^2}\big(\vert k_1 \vert \vert k_2 \vert \vert k_1^{*} \vert \vert k_2^{*} \vert \big)^{2\beta}|k_1-k_2|g(k_1^*)h(k^*_2) b\left(\left(k_1-k_2\right)\cdot\sigma\right) \,d\sigma\,dk_2    \label{G_2 parametrized}\\
\mathcal{L}_1[f,g,h]&=\frac{1}{4}f(k)\int_{\R^3\times\S^2}\big(\vert k_1 \vert \vert k_2 \vert \vert k_1^{*} \vert \vert k_2^{*} \vert \big)^{2\beta}|k_1-k_2|g(k_2)h(k^*_2) b\left(\left(k_1-k_2\right)\cdot\sigma\right) \,d\sigma\,dk_2    \label{L_2 parametrized}\\
\mathcal{L}_2[f,g,h]&=\frac{1}{4}f(k_1)\int_{\R^3\times\S^2}\big(\vert k_1 \vert \vert k_2 \vert \vert k_1^{*} \vert \vert k_2^{*} \vert \big)^{2\beta}|k_1-k_2|g(k_2)h(k_1^*) b\left(\left(k_1-k_2\right)\cdot\sigma\right) \,d\sigma\,dk_2    \label{L_1 parametrized},
\end{align}
where given $\sigma\in \S^2$ we denote
\begin{align}
 k_1^*&=\frac{k_1+k_2}{2}+\frac{|k_1-k_2|}{2}\sigma\label{k_1^*}\\
 k_2^*&=\frac{k_1+k_2}{2}-\frac{|k_1-k_2|}{2}\sigma\label{k_2^*}
\end{align}
It is straightforward to check that $k_1,k_2,k_1^*,k_2^*$ provide the unique solution the conservation of momentum-energy system
\begin{align}
k_1^*+k_2^*&=k_1+k_2,\label{conservation of momentum}\\ 
|k_1^*|^2+|k_2^*|^2&=|k_1|^2+|k_2|^2,\label{conservation of energy}
\end{align}
which is draws an analogy between interactions of waves and elastic collisions of particles.
\subsection{Outline of the paper}
We first prove local well-posedness for $\beta \leq \frac{1}{4}$ by a fixed-point argument in Section \ref{section::LWP}. This largely relies on tools from our earlier work \cite{AmLe}, particularly the collision averaging estimate proved therein. 

To show ill-posedness above the threshold, we exhibit initial datum that lead to immediate loss of smoothness. Since all such examples are radial we work with the angularly averaged version of the equation, which we derive in Section \ref{isotropic}. We prove that both the gain-only and the full equation are ill-posed above the threshold in Sections \ref{gain:illposed} and \ref{illposed:full} respectively. We remark however that the example used for the gain only equation does not lead to ill-posedness for the full equation due to some cancellations in the nonlinearity. As a result the choice of initial datum is more subtle in the latter case.

\subsection{Notations}
We will write $A \lesssim B$ if there exists a numerical constant $C$ such that $A \leq C B.$ We will write $A \approx B$ if $A \lesssim B$ and $B \lesssim A.$ 

We use the standard japanese bracket notation $\langle x \rangle := \sqrt{1 + \vert x \vert^2}.$ 

For $v \in \mathbb{R}^3$ we write $\widehat{v} := \frac{v}{\vert v \vert}.$ 

\subsection*{Acknowledgements}
I.A. was supported by the NSF grant No. DMS-2418020. T.L. was supported by the Simons grant on wave turbulence. The authors thank Jalal Shatah for suggesting this problem as well as his comments on an earlier version of this manuscript. They also thank Gregory Falkovich for his remarks regarding the physical interpretation of their results.

\section{Well-posedness for $0 \leq \beta \leq \frac{1}{4}$} \label{section::LWP}
For this part of the paper, we follow along the lines of the strategy introduced in our earlier paper \cite{AmLe}. We start by recalling some technical results, chief among which our angular averaging lemma \ref{singint}. We deduce trilinear bounds for the collision operator on weighted $L^\infty$ spaces and use them to prove the desired local well-posedness result.

\subsection{Preliminary results}

The main tool to overcome the growth due to the cross-section is the following collisional averaging estimate proved in \cite{AmLe}. 
 \begin{lemma}[\cite{AmLe}] \label{singint}
 Let
\begin{align*}
F(k_1,k_2) : = \int_{\S^{2}}\frac{1}{\l k_1^*\r^{3}}\,d\sigma,\quad F_1(k_1,k_2) : = \int_{\S^{2}}\frac{1}{\l k_2^*\r^{3}}\,d\sigma 
\end{align*}
where 
\begin{align*}
  k_1^*&=\frac{k_1+k_2}{2}+\frac{|k_1-k_2|}{2}\sigma,\\
  k_2^*&=\frac{k_1+k_2}{2}-\frac{|k_1-k_2|}{2}\sigma
\end{align*}
Then, we have for all $k_1,k_2 \in \mathbb{R}^3$
\begin{align}\label{L inf bound}
F(k_1,k_2),\,F_1(k_1,k_2)\lesssim \frac{1}{1 + \vert k_1 \vert^2 + \vert k_2 \vert^2} .
\end{align}
\end{lemma}

We also recall the following useful $L^1$ estimates 
\begin{lemma}[\cite{AmLe}]\label{pre average lemma} The following estimates hold
\begin{align}
\sup_{k_1\in\R^3}\int_{\R^3\times \S^2}\,|\psi(k_2^*)|b\left(\frac{k_1-k_2}{|k_1-k_2|}\cdot\sigma\right)\,dk_2\,d \sigma&\lesssim \|\psi\|_{L^{1}}\label{estimate for v_1 pre},\\
\sup_{k_2\in\R^3}\int_{\R^3\times \S^2}\,|\psi(k_1^*)|b\left(\frac{k_1-k_2}{|k_1-k_2|}\cdot\sigma\right)\,dk_1\,d \sigma&\lesssim \|\psi\|_{L^{1}}.\label{estimate for v pre}
\end{align}
\end{lemma}

\subsection{Trilinear bounds}
We deduce several trilinear bounds for collision operators from the facts recalled in the previous section. The proofs follow our earlier paper \cite{AmLe} closely. 

\begin{lemma}
Assume that $0 \leq \beta \leq 1/4.$ 
For $\mathcal{T} \in \lbrace \mathcal{L}_1, \mathcal{L}_2 \rbrace$ and $l \geq 0$, we have

\begin{align} \label{trilossmain}
    \Vert \langle k_1 \rangle^l \mathcal{T}[f,g,h] \Vert_{L^{\infty}_{k_1}} \lesssim \Vert \langle k_1 \rangle^l f \Vert_{L^\infty_{k_1}} \Vert \langle k_1 \rangle^{1/2} g \Vert_{L^1_{k_1}} \Vert \langle k_1 \rangle^{7/2} h \Vert_{L^{\infty}_{k_1}}.
\end{align}
We also have
\begin{equation}\label{G_2 weighted estimate}
\|\l k_1\r^l \mathcal{G}_2[f,g,h]\|_{L^\infty_{k_1}}\leq \|\l k_1 \r^l f\|_{L^\infty_{k_1}}\|\l k_1 \r^3 g\|_{L^\infty_{k_1}}\|\l k_1\r^3 h\|_{L^1_{k_1}} . 
\end{equation}
Finally
\begin{align}\label{G1 weighted estimate}
    \Vert \langle k_1 \rangle^l \mathcal{G}_1[f,g,h] \Vert_{L^{\infty}_{k_1}}\lesssim \Vert \l k_1 \r^{1/2} f \Vert_{L^1_{k_1}} \Vert \langle k_1 \rangle^l g \Vert_{L^{\infty}_{k_1}} \Vert \langle k_1 \rangle^{7/2} h \Vert_{L^{\infty}_{k_1}}+  \|\l k_1 \r^l f\|_{L^\infty_{k_1}}\|\l k_1\r^3 g\|_{L^\infty_{k_1}}\|\l k_1\r^3 h\|_{L^1_{k_1}}
\end{align}
\end{lemma}

\begin{proof}
We write
\begin{align*}
    \Vert \langle k_1 \rangle^l \mathcal{L}_1[f,g,h] \Vert_{L^{\infty}_{k_1}} & \lesssim \Vert \langle k_1 \rangle^l f \Vert_{L^\infty_{k_1}} \Bigg \Vert \int_{\mathbb{R}^3 \times \mathbb{S}^{2}} \vert k_2 \vert^{2 \beta} g(k_2) \vert k_1 \vert^{2 \beta} \vert k_1^* \vert^{2 \beta} \vert k_1-k_2 \vert \vert k_2^{*} \vert^{2\beta} h(k_2^{*}) \,  d\sigma dk_2  \Bigg \Vert_{L^{\infty}_{k_1}}  \\
    & \lesssim \Vert \langle k_1 \rangle^l f \Vert_{L^\infty_{k_1}} \Vert \vert k_1 \vert^{1/2} g \Vert_{L^1_{k_1}} \Bigg \Vert \int_{\S^{2}} \langle k_2^{*} \rangle^{3+2\beta} h(k_2^{*}) \frac{\vert k_1 - k_2 \vert \vert k_1 \vert^{2\beta} \vert k_1^* \vert^{2 \beta}  }{\langle k_2^{*} \rangle^3}  d\sigma \Bigg \Vert_{L^{\infty}_{k_1,k_2}} \\
    & \lesssim \Vert \langle k_1 \rangle^l f \Vert_{L^\infty_{k_1}} \Vert \vert k_1 \vert^{1/2} g \Vert_{L^1_{k_1}} \Vert \langle k_1 \rangle^{7/2} h \Vert_{L^\infty_{k_1}} \Bigg \Vert \frac{(\vert k_1 \vert + \vert k_2 \vert) \vert k_1 \vert^{2\beta} \vert k_1^* \vert^{2 \beta} }{1 + \vert k_1 \vert^2 + \vert k_2 \vert^2} \Bigg \Vert_{L^{\infty}_{k_1,k_2}}.
\end{align*}
Next write that using Lemma \ref{singint}
\begin{align*}
\frac{(\vert k_1 \vert + \vert k_2 \vert) \vert k_1 \vert^{2\beta} \vert k_1^* \vert^{2 \beta} }{1 + \vert k_1 \vert^2 + \vert k_2 \vert^2} \leq \frac{\vert k_1^* \vert^{2\beta}}{\big(1 + \vert k_1 \vert^2 + \vert k_2 \vert^2 \big)^{1/4}} \leq \frac{\vert k_1^* \vert^{2\beta}}{\big(1 + \vert k_1^* \vert^2 + \vert k_2^* \vert^2 \big)^{1/4}} \leq 1,
\end{align*}
where for the penultimate inequality we used energy conservation. 

The proof for $\mathcal{L}_2$ is identical, substituting $k_2^{*}$ with $k_1^{*}.$ 

For the second estimate we write that, denoting $\widetilde{g}(k_1)=\l k_1 \r^3 g(k_1)$, $\widetilde{h}(k_1)=\l k_1\r^3 h(k_1)$:
\begin{align*}
   & \|\l k_1\r^l\mathcal{G}_2[f,g,h]\|_{L^\infty_{k_1}}\\
   &\lesssim\|\l k_1\r^l f\|_{L^\infty_{k_1}}\left\|\int_{\R^3\times\S^2}\frac{|k_1-k_2| \vert k_1 \vert^{2\beta} \vert k_2 \vert^{2\beta}}{\l k_1^*\r ^2 \l k_2^*\r ^2 } \frac{\vert k_1^* \vert^{2\beta} \vert k_2^* \vert^{2\beta}}{\l k_1^*\r \l k_2^*\r }  \widetilde{g}(k_1^*)\widetilde{h}(k_2^*)b\left(\frac{k_1-k_2}{|k_1-k_2|}\cdot\sigma\right)\,dk_2\,d\sigma\right\|_{L^\infty_{k_1}}\\
    & \leq \|\l k_1\r^l f\|_{L^\infty_{k_1}} \|\widetilde{g}\|_{L^\infty_{k_1}} \left\|\int_{\R^3\times\S^2}\frac{|k_1-k_2| \vert k_1 \vert^{2\beta} \vert k_2 \vert^{2\beta}}{\l k_1^*\r^2 \l k_2^*\r^2} \vert \widetilde{h}(k_2^*) \vert b\left(\frac{k_1-k_2}{|k_1-k_2|}\cdot\sigma\right)\,dk_2\,d\sigma\right\|_{L^\infty_{k_1}}\\
    &\lesssim \|\l k_1\r^l f\|_{L^\infty_{k_1}} \|\widetilde{g}\|_{L^\infty_{k_1}} \left\|\int_{\R^3\times\S^2} \vert \widetilde{h}(k_2^*) \vert b\left(\frac{k_1-k_2}{|k_1-k_2|}\cdot\sigma\right)\,dk_2\,d\sigma\right\|_{L^\infty_{k_1}}\\
    &\leq \|\l k_1\r^l f\|_{L^\infty_{k_1}} \|\widetilde{g}\|_{L^\infty_{k_1}}\|\widetilde{h}\|_{L^1_{k_1}},
\end{align*}
where in the last two lines we used the fact that by conservation of momentum and energy
$$\frac{|k_1-k_2| \vert k_1 \vert^{2\beta} \vert k_2 \vert^{2\beta}}{\l k_1^*\r^2  \l k_2^*\r^2 } \leq \frac{|k_1-k_2| \big( \vert k_1 \vert^{2} + \vert k_2 \vert^2 \big)^{2\beta}}{\l k_1^*\r^2  \l k_2^*\r^2 } = \frac{|k_1^*-k_2^*| \big( \vert k_1^* \vert^{2} + \vert k_2^* \vert^2 \big)^{2\beta}}{\l k_1^*\r^2  \l k_2^*\r^2 } \lesssim 1 ,$$
as well as Lemma \ref{pre average lemma}.

\end{proof}

\begin{corollary}
For $\mathcal{T} \in \lbrace \mathcal{G}_1, \mathcal{G}_2, \mathcal{L}_1, \mathcal{L}_2 \rbrace$ and $l>6$, we have
\begin{align} \label{LWP estimate}
    \Vert \langle k_1 \rangle^l \mathcal{T}[f,g,h] \Vert_{L^{\infty}_{k_1}} \lesssim \Vert \l k_1 \r^l f \Vert_{L^\infty_{k_1}} \Vert \l k_1 \r^l g \Vert_{L^\infty_{k_1}} \Vert \l k_1 \r^l h \Vert_{L^\infty_{k_1}}.
\end{align}
\end{corollary}
\begin{proof}
The result directly follows from the previous lemma and the fact that
\begin{align*}
    \Vert \l k_1 \r^3 f \Vert_{L^1_{k_1}} \lesssim \Vert \l k_1 \r^{-3^{-}} \Vert_{L^1_{k_1}}  \Vert \l k_1 \r^{6^+} f \Vert_{L^\infty_{k_1}}.
\end{align*}
\end{proof}

\subsection{Strong solutions} \label{LWP}
In this section we construct local strong solutions to \eqref{DKWE}. To make matters more precise, we introduce the notion of a strong solution to \eqref{DKWE}:

\begin{definition}\label{definition of a solution} 
Let $T>0$ and $f_0\in \langle k_1 \rangle^{-M} L^\infty_{k_1}$. We say that $f\in \mathcal{C}([0,T];\l k_1\r^{-M}L^\infty_{k_1})$ is a solution of \eqref{DKWE} on $[0,T]$ with initial datum $f_0$, if
\begin{equation}\label{Duhamel formula}
f(t)=f_0+\int_0^t \mathcal{C}[f](s)\,ds,\quad  
t\in [0,T].
\end{equation}  
\end{definition}

\begin{theorem}
If $0 \leq \beta \leq \frac{1}{4},$ then for all $M>6,$ \eqref{DKWE} is locally well-posed for $f_0 \in \l k_1 \r^{-M} L^{\infty}_{k_1}.$ 
\end{theorem}
\begin{remark}
    Given that the free transport is an isometry on the spaces $\langle k \rangle^{-M} L^\infty_{xk},$ the proof applies to the inhomogeneous equation as well. 
\end{remark}
\begin{proof}
Let $C_1 > 0$ be the numerical constant in \eqref{LWP estimate}. Let $R := \Vert \l k_1 \r^M f_0 \Vert_{L^{\infty}_{k_1}},$ $T:= \frac{1}{8 C_1 R^2}$ and  
\begin{align*}
\textbf{B}(2R) := \bigg \lbrace f \in \mathcal{C} \big( [0,T]; \l k_1 \r^{-M} L^{\infty}_{k_1} \big) ; \sup_{t \in [0;T]} \Vert \l k_1 \r^{M} f(t) \Vert_{L^{\infty}_{k_1}} \leq 2R \bigg \rbrace.
\end{align*}

Consider the map
\begin{align*} \Phi: 
    \begin{cases}
    \textbf{B}(2R) & \longrightarrow \, \mathcal{C} \big( [0,T]; \l k_1 \r^{-M} L^{\infty}_{k_1} \big) \\
    f & \longmapsto \,  f_0 + \int_0^t \mathcal{C}[f](s) \, ds
    \end{cases} .
\end{align*}
We will prove that $\Phi$ is a contraction on $\textbf{B}(2R).$

\underline{Stability:} Let $f \in \textbf{B}(2R).$ Using \eqref{LWP estimate} we find
\begin{align*}
    \Vert \l k_1 \r^M \Phi (f) \Vert_{L^\infty_{k_1}} \leq R + C_1 T (2R)^3 = 2R .
\end{align*}

\underline{Contraction:} Let $f,g \in \textbf{B}(2R).$ Using \eqref{LWP estimate} we find

\begin{align*}
    \big \Vert \l k_1 \r^M \big( \Phi (f) - \Phi(g) \big) \big \Vert_{L^\infty_{k_1}} \leq 4 C_1 R^2 T \big \Vert \l k_1 \r^M \big( f-g \big) \big \Vert_{L^\infty_{k_1}} \leq \frac{1}{2}  \Vert \l k_1 \r^M \big( f-g \big) \big \Vert_{L^\infty_{k_1}} . 
\end{align*}
The proof is complete by the contraction mapping theorem.

\end{proof}

We now move to the main part of the paper where we prove our ill-posedness results.

\section{Isotropic equation} \label{isotropic}
The initial datum leading to ill-posedness will be isotropic i.e. $f^0(k_1) = n^0(\vert k_1 \vert^2)$ for some function $n^0:[0,+\infty)\to\R$. In this section, we first show that the evolution of isotropic datum remains isotropic and derive the equation it satisfies. Then, we further decompose the isotropic collisional operator based on the value of the cross-section.

\subsection{Derivation}
The main result of this section is 
\begin{lemma} \label{isotropic-full}
Let $f^0 \in \langle k_1 \rangle ^{-M} L^\infty_{k_1}$ be isotropic.
Assume that \eqref{DKWE} has a unique strong solution  $f(t,k_1)$ (in the sense of Definition \ref{definition of a solution}) in $[0,T]$ with initial datum $f^0$. Then, denoting $\omega_1:=\vert k_1 \vert^2,$ there exists unique $n \in \mathcal{C} \big([0,T];\langle \omega_1 \rangle^{-M/2} L^\infty_{\omega_1}([0,\infty))\big)$ such that $n(t,\omega_1) = f(t,k_1).$ 
Moreover $n(t,\omega_1)$ is the unique strong solution \footnote{The definition of a strong solution to \eqref{iso-DKWE} is given analogously to Definition \ref{definition of a solution}} of
\begin{equation}\label{iso-DKWE}
    \begin{cases}
        \partial_tn=\mathcal{C}[n]\\
        n(0)=n^0
    \end{cases},\quad t\in[0,T],
\end{equation}
where $n^0(\omega_1)=f^0(k_1)$. The isotropic collisional operator  $\mathcal{C}$ (we slightly abbreviate notation still denoting it by $\mathcal{C}$) is given by
\begin{equation}\label{iso collisional operator}
\begin{aligned}  
    \mathcal{C}[n] &= \int_{0 \leq \omega_3 \leq \omega_4, \,\omega_2\geq 0} \hspace{-.3cm}S(\omega_1,\omega_2,\omega_3,\omega_4) \, \left(n_2n_3n_4+n_1n_3n_4-n_1n_2n_3-n_1n_2n_4\right)
     d\omega_3 d\omega_4,\\
     &\omega_2=\omega_4+\omega_3-\omega_1,\\
      &S(\omega_1,\omega_2,\omega_3,\omega_4)=64 \pi^3 \omega_1^{\beta-1/2} (\omega_2 \omega_3 \omega_4)^{\beta} \min \lbrace \sqrt{\omega_1}, \sqrt{\omega_2},\sqrt{\omega_3}\rbrace,
\end{aligned}
\end{equation}
 and above  we denote $n_i:=n(t,\omega_i)$ for $i\in\{1,2,3,4\}$.
\end{lemma}
\begin{remark} \label{rmk-inhomo}
A similar statement can be shown for solutions to the inhomogeneous equation.
\end{remark}

\begin{proof}
% Since the initial datum $f^0$ is isotropic, there exists $n^0:[0,\infty)\to\R$ such that $f^0(k_1)=n^0(\omega_1)$.
We observe that if $f(t,k_1)$ is a solution to \eqref{DKWE}, then for any rotation $R \in SO(3),$ $f(t,Rk_1)$ also solves \eqref{DKWE}. This can be seen from the parametrization \eqref{G_1 parametrized}-\eqref{L_1 parametrized} by changing variables and using invariance of the norm by rotation. Since $f^0$ is isotropic, both functions are coincide at time 0, hence by  uniqueness we have $f(t,Rk_1) = f(t,k_1)$ for all $t \in [0;T].$ Therefore, there exists $n \in \mathcal{C}\big([0,T];\langle \omega_1 \r ^{-M/2} L^\infty_{\omega_1}([0,\infty))\big)$ such that $f(t, k_1)=n(t,\omega_1)$, where we denote $\omega_1:=|k_1|^2$. 

We now derive the equation satisfied by $n(t,\omega_1)$ following \cite{ZSKN}. We use spherical coordinates to write $dk_j = \sqrt{\omega_j} d \omega_j d\Theta_j,$ with $d\Theta_j$ the measure on the unit sphere and $\omega_j:=|k_j|^2$.

Next, exploiting that $n$ is isotropic we can average over the angle and use Fubini's theorem to obtain
\begin{align*}
    n(t,\omega_1) &=\frac{1}{4 \pi}\int_{\mathbb{S}^2}  n(t,\omega_1) d\Theta_1 \\
     &=n^{0}(\omega_1) +\int_0^t\omega_1^{\beta/2}\int_{[0,\infty)^3}  \mathcal{P} n_1 n_2 n_3 n_4 \big( \frac{1}{n_1} + \frac{1}{n_2} -\frac{1}{n_3} - \frac{1}{n_4}   \big) (\omega_2 \omega_3 \omega_4)^{\frac{1+\beta}{2}} \delta(\Omega) d\omega_2 d\omega_3 d\omega_4\,ds
\end{align*}
with 
\begin{align*}
    \mathcal{P} = \int_{\mathbb{S}^8} \delta(\Sigma) d\Theta_1 d\Theta_2 d\Theta_3 d\Theta_4 = \frac{32 \pi^3}{\sqrt{\omega_1 \omega_2 \omega_3 \omega_4}} \min \lbrace \sqrt{\omega_1},\sqrt{\omega_2}, \sqrt{\omega_3}, \sqrt{\omega_4} \rbrace,
\end{align*}
where the second equality can be found in \cite{ZSKN}, Appendix A.

% Thus
% \begin{equation}\label{iso-KWE}
% \begin{aligned}
%  &\partial_t n_1 = \frac{1}{2} \int_{\omega_3, \omega_4 \geq 0 } \widetilde{S}(\omega_1,\omega_2,\omega_3,\omega_4) \, n_1 n_2 n_3 n_4 \big(\frac{1}{n_1} + \frac{1}{n_2} - \frac{1}{n_3} - \frac{1}{n_4}
%     \big) d\omega_3 d\omega_4, \\
%  &\omega_2 = \omega_3 + \omega_4 - \omega_1, \\
%   &   \widetilde{S}(\omega_1,\omega_2,\omega_3,\omega_4)= 64 \pi^3 \omega_1^{\beta-1/2} (\omega_2 \omega_3 \omega_4)^{\beta} \min \lbrace \sqrt{\omega_1}, \sqrt{\omega_2},\sqrt{\omega_3},\sqrt{\omega_4} \rbrace,
% \end{aligned}
% \end{equation}
Using symmetry of the integrand as well, the equation can be equivalently written in the form \eqref{iso-DKWE}.
\end{proof}
\subsection{Decomposition of the isotropic collision operator}\label{subsec:domain} We first decompose the collisional operator as 
\begin{equation}
\mathcal{C}[n]= \mathcal{C}^{234}[n,n,n]+\mathcal{C}^{134}[n,n,n]-\mathcal{C}^{123}[n,n,n]-\mathcal{C}^{124}[n,n,n],
\end{equation}
where given  $\bm{j}=j_1j_2j_3\in\{234,134,123,124\}$, we denote
\begin{equation}\label{C^ijk}
\mathcal{C}^{\bm{j}}[k,l,m]:= \int_{0\leq\omega_3\leq\omega_4,\,\omega_2\geq 0}S(\omega_1,\omega_2,\omega_3,\omega_4)k_{j_1}l_{j_2}m_{j_3}\,d\omega_3\,d\omega_4, 
\end{equation}
and we recall that $\omega_2=\omega_4+\omega_3-\omega_1$.

Then we further decompose each $\mathcal{C}^{\bm{j}}$ based on the value of the cross-section $S$ as follows
\begin{equation}
\mathcal{C}^{\bm{j}}[k,l,m]=\mathcal{C}_1^{\bm{j}}[k,l,m]+\mathcal{C}_2^{\bm{j}}[k,l,m]+\mathcal{C}_3^{\bm{j}}[k,l,m],    
\end{equation}
where given $i\in\{1,2,3\}$ we denote
\begin{equation}\label{operator i j}
\mathcal{C}_i^{\bm{j}}[k,l,m]:=64\pi^3\omega_1^{\beta-1/2}\int_{0\leq\omega_3\leq\omega_4,\,\omega_2\geq 0}\sqrt{\omega_i}(\omega_2\omega_3\omega_4)^\beta\, \mathds{1}_{\omega_i=\min\{\omega_1,\omega_2,\omega_3\}}k_{j_1}l_{j_2}m_{j_3}\,d\omega_3\,d\omega_4.    
\end{equation}
We note that the operators $\mathcal{C}_i^{\bm{j}}$ are multilinear and monotone for positive inputs. 

We now determine the domain of integration corresponding to each $i\in\{1,2,3\}$.

% We will use the notation $\mathcal{C}^{j}_i$ for $j \in \lbrace 123,234,134, 124 \rbrace$ and $i \in \lbrace 1,2,3 \rbrace$ to denote the operator with nonlinear term $n_{j_1} n_{j_2} n_{j_3}$ where $j = j_1 j_2 j_3, j \in \lbrace 123,234,134, 124 \rbrace$ and an additional cut-off $\textbf{1}_{\omega_i = \min \lbrace \omega_1,\omega_2,\omega_3 \rbrace}$ in the integrand.

\subsubsection{Operator $\mathcal{C}_2^{\bm{j}}$.} \label{domain:C2} 
In this case, the condition $\omega_2 \leq \omega_3$ implies $\omega_4 \leq \omega_1.$ Moreover the condition $\omega_2 \geq 0$ gives $\omega_4 \geq \omega_1-\omega_3.$ Since we also have $\omega_4 \geq \omega_3,$ we  further split the operator $\mathcal{C}_2^{\bm{j}}$  into $\mathcal{C}_{21}^{\bm{j}}$ and $\mathcal{C}_{22}^{\bm{j}}$ with corresponding domains of integration
\begin{align}
    \mathcal{D}_{21}&:=\left\{(\omega_3,\omega_4)\in [0,\infty)^2\,:\, 0\leq\omega_3\leq\frac{\omega_1}{2},\,\,\omega_1-\omega_3\leq\omega_4\leq\omega_1 \right\}\label{D21}\\
    \mathcal{D}_{22}&:=\left\{(\omega_3,\omega_4)\in [0,\infty)^2\,:\, \frac{\omega_1}{2}\leq\omega_3\leq\omega_4\leq\omega_1 \right\}\label{D22}.
\end{align}
That is 
\begin{equation}\label{C_2^j}
    \mathcal{C}_2^{\bm{j}}[k,l,m]=\mathcal{C}_{21}^{\bm{j}}[k,l,m]+\mathcal{C}_{22}^{\bm{j}}[k,l,m],
\end{equation}
where 
\begin{align}
    \mathcal{C}_{21}^{\bm{j}}[k,l,m]&:=64\pi^3\omega_1^{\beta-\frac{1}{2}}\int_{\mathcal{D}_{21}}\omega_2^{\beta+\frac{1}{2}}\omega_3^\beta \omega_4^\beta\, k_{j_1}l_{j_2}m_{j_3}\,d\omega_3\,d\omega_4,\label{C_21^j}\\
    \mathcal{C}_{22}^{\bm{j}}[k,l,m]&:=64\pi^3\omega_1^{\beta-\frac{1}{2}}\int_{\mathcal{D}_{22}}\omega_2^{\beta+\frac{1}{2}}\omega_3^\beta \omega_4^\beta\, k_{j_1}l_{j_2}m_{j_3}\,d\omega_3\,d\omega_4.\label{C_22^j}
\end{align}

\subsubsection{Operator $\mathcal{C}_3^{\bm{j}}$.} \label{domain:C3}
Here the condition $\omega_3 \leq \omega_2$ gives $\omega_1 \leq \omega_4.$ Then the condition $\omega_4 \geq \omega_1 - \omega_3$ is automatically satisfied, and since $\omega_3 \leq \omega_1$ we also have $\omega_4 \geq \omega_3.$ Therefore, the operator $\mathcal{C}_3^{\bm{j}}$ is expressed as
\begin{equation}\label{C_3^j}
\mathcal{C}^{\bm{j}}_3[k,l,m]=64\pi^3\omega_1^{\beta-\frac{1}{2}}\int_{\mathcal{D}_3}\omega_2^\beta\omega_3^{\beta+\frac{1}{2}}\omega_4^\beta\, k_{j_1}l_{j_2}m_{j_3}\,d\omega_3\,d\omega_4,    
\end{equation}
where the domain of integration is given by
\begin{align}
    \mathcal{D}_3&:=\left\{(\omega_3,\omega_4)\in[0,\infty)^2\,: 0\leq\omega_3\leq\omega_1\leq\omega_4\right\}\label{D3}.
\end{align}

\subsubsection{Operator $\mathcal{C}_1^{\bm{j}}$.} \label{domain:C1}
This part will always be a lower order term. The condition $\omega_2 \geq 0$ is always satisfied. Therefore, the operator $\mathcal{C}_1^{\bm{j}}$ is expressed as
\begin{equation}\label{C_1^j}
\mathcal{C}^{\bm{j}}_1[k,l,m]=64\pi^3\omega_1^\beta\int_{\mathcal{D}_1}\omega_2^\beta\omega_3^\beta\omega_4^\beta\, k_{j_1}l_{j_2}m_{j_3}\,d\omega_3\,d\omega_4,    
\end{equation}
where the domain of integration is 
\begin{align}
    \mathcal{D}_1&:=\left\{(\omega_3,\omega_4)\in[0,\infty)^2\,: 0\leq\omega_1\leq\omega_3\leq\omega_4\right\}\label{D1}.
\end{align}

\section{Ill-posedness for the gain-only equation \eqref{DKWE} if $\beta>\frac{1}{4}$} \label{gain:illposed}

We now move to the gain-only equation \eqref{DKWE} $\partial_t f = \mathcal{G}[f]$ (recall that $\mathcal{G}$ is defined in \eqref{gain def}). Our notion of strong solution in this setting is as in Definition \ref{definition of a solution} with $\mathcal{C}$ replaced by $\mathcal{G}.$

Through a straightforward adaptation of the proof, we have the following gain-only version of Lemma \ref{isotropic-full}
\begin{lemma} \label{isotropic-gain}
Let $f^0 \in \langle k \rangle ^{-M} L^\infty_k$ be isotropic.
Assume that the gain-only equation \eqref{DKWE} with initial datum $f^0$ has a unique strong solution on $[0,T].$ Then there exists a  unique function $n \in \mathcal{C}([0,T];\langle \omega_1 \rangle^{-M/2} L^\infty_{\omega_1}([0,+\infty))$ such that $n(t,\omega_1) = f(t,k_1).$ 
Moreover $n(t,\omega_1)$ is the unique strong solution of
\begin{equation}\label{gain IVP}
\begin{cases}
    \partial_t n = \mathcal{C}^{234}[n,n,n] + \mathcal{C}^{134}[n,n,n]\\
    n(0)=n^0
    \end{cases},\quad t\in[0,T],
\end{equation}
 where $n^0(\omega_1)=f^0(k_1)$ and $\mathcal{C}^{234},\mathcal{C}^{134}$ are given by \eqref{C^ijk}.
\end{lemma}

% \subsection{Technical lemma}
% We will make repeated use of the following lemma in what follows.
% \begin{lemma} \label{technical}
% Let $A,B>1$ and $0 < a\leqslant b.$ Then 
% \begin{align*}
%     \int_0^{\infty} \frac{1}{1+(a+x)^A} \frac{1}{1+(b+x)^B} dx \approx \frac{1}{1+a^{A-1}} \frac{1}{1+b^B}.
% \end{align*}
% \end{lemma}
% \begin{proof}
% Split the integral in three regions:

% \underline{Case 1: $x \in [0;a]$.} Then 
% \begin{align*}
% \frac{1}{1+(a+x)^A} \frac{1}{1+(b+x)^B} \approx \frac{1}{1+a^A} \frac{1}{1 + b^B}
% \end{align*}
% and the result follows.

% \underline{Case 2: $x \in [a;b]$.} Then 

% \begin{align*}
% \int_a^{b} \frac{1}{1+(a+x)^A} \frac{1}{1+(b+x)^B} dx \approx \frac{1}{1+b^B} \int_a^b \frac{1}{1+(a+x)^A} dx
% \end{align*}
% and the result follows.

% \underline{Case 3: $x \in [b;+\infty)$.} Then 
% \begin{align*}
% \int_b^{\infty} \frac{1}{1+(a+x)^A} \frac{1}{1+(b+x)^B} dx \approx \frac{1}{1+b^A} \int_b^{\infty} \frac{1}{(b+x)^B} dx
% \end{align*}
% and we can conclude.
% \end{proof}

% {\color{blue}
% \comment{IA Can't we write as an estimate on $$\int_0^\infty \frac{1}{\l a+x\r^A}\frac{1}{\l b+x\r^B}\,dx$$
% I think it will be easier to apply. Also clarify the size of $a$? Maybe also state the  three different cases separately because we care about the lower bound

% }
% }

We now state our first ill-posedness result.

\begin{theorem} \label{thm:gain-only}
Assume that $\beta>\frac{1}{4}.$ Then the gain-only equation \eqref{DKWE} is ill-posed in all the spaces $\l k_1 \r^{-M} L^{\infty}_{k_1}$ for $M>6.$ That is, there exists some initial datum in $\l k_1 \r^{-M} L^{\infty}_{k_1}$ such that no strong solution can be constructed. 
\end{theorem}
\begin{remark}
As will be clear from the proof, the equation is also ill-posed in the small data regime since the initial datum leading to uniqueness can be scaled.
\end{remark}

The strategy will be to show that second Picard iterate is unbounded. It suffices to show a lower bound for $\mathcal{C}^{234}_{21}[n^{01},n^{01},n^{01}]$ for $\omega_1>1$.

\subsection{Lower bound for $\mathcal{C}^{234}[n^{01},n^{01},n^{01}]$} 
Given $\omega_1>1$, in the domain $\mathcal{D}_{21}$ we have $\omega_4\approx \omega_1$, so for the corresponding inner integral we obtain
\begin{align*}
 \int_{\omega_1-\omega_3}^{\omega_1}   \frac{\omega_2 ^{\beta+\frac{1}{2}}  \omega_4 ^\beta}{\langle \omega_2 \rangle^{M/2} \l  \omega_4 \rangle^{M/2}} d\omega_4 
 &\approx \omega^{\beta-M/2}_1\int_0^{\omega_3} \frac{\omega_2^{\beta+\frac{1}{2}}}{\l \omega_2\r^{M/2}}\,d\omega_2\\
 &\gtrsim \frac{\omega^{\beta-M/2}_1}{\l \omega_3\r^{M/2}}\int_0^{\omega_3}\omega_2^{\beta+\frac{1}{2}}\,d\omega_2\\
 &\approx \omega_1^{\beta-M/2}\frac{\omega_3^{\beta+\frac{3}{2}}}{\l \omega_3\r^{M/2} }.
\end{align*}
Hence, we conclude
\begin{align}
\label{lower bound C^234}\mathcal{C}^{234}[n^{01},n^{01},n^{01}]\geq \mathcal{C}^{234}_{21}[n^{01},n^{01},n^{01}]
 \gtrsim  \omega_1^{2\beta-\frac{1}{2}-M/2}\int_0^{1/2}\frac{\omega_3^{2\beta+\frac{3}{2}}}{\l \omega_3\r^{M}}\,d\omega_3
 \approx \omega_1^{2\beta-\frac{1}{2}-M/2}.
\end{align}

Strictly speaking, the previous computation is sufficient to prove ill-posedness of the gain-only equation. We nonetheless compute the exact size of $\mathcal{C}^{234}[n^{01}]$ since it will  be useful in the treatment of the full equation, see Section \ref{illposed:full}.

\subsection{Upper bound for $\mathcal{C}^{234}[n^{01},n^{01},n^{01}]$ } For this we will provide an upper bound for each of $\mathcal{C}_{21}^{234}$, $\mathcal{C}_{22}^{234}$, $\mathcal{C}_{3}^{234}$,  $\mathcal{C}_{1}^{234}$ for $\omega_1>1$.

\begin{itemize}
    \item Upper bound for $\mathcal{C}_{21}^{234}[n^{01},n^{01},n^{01}]$:  Given $\omega_1>1$, in the domain $\mathcal{D}_{21}$ we have $\omega_2\leq\omega_3$ and $\omega_4\approx\omega_1$, therefore for the inner integral we have
\begin{align*}
\int_{\omega_1-\omega_3}^{\omega_1}\frac{\omega_2^{\beta+\frac{1}{2}}\omega_4^\beta}{\l\omega_2\r^{M/2}\l\omega_4\r^{M/2}}\,d\omega_4\lesssim\omega_1^{\beta-M/2}\omega_3^{\beta+\frac{1}{2}}\int_{\omega_3}^\infty\frac{1}{\l x\r^{M/2}}\,dx\lesssim  \omega_1^{\beta-M/2}\omega_3^{\beta+\frac{1}{2}}.   
\end{align*}
Hence, we obtain
\begin{align}\label{UB C_21^234}
\mathcal{C}_{21}^{234}[n^{01},n^{01},n^{01}]\lesssim \omega_1^{2\beta-\frac{1}{2}-M/2}\int_{0}^{\omega_1/2}\frac{1}{\l\omega_3\r^{M/2-\beta-\frac{1}{2}}}\,d\omega_3 \lesssim \omega_1^{2\beta-\frac{1}{2}-M/2}.   
\end{align}

\item Upper bound for $\mathcal{C}_{22}^{234}[n^{01},n^{01},n^{01}]$: Given $\omega_1>1$, in the domain $\mathcal{D}_{22}$ we have $\omega_3\approx\omega_4\approx\omega_1$, so for the inner integral we have
\begin{align*}
\int_{\omega_3}^{\omega_1}\frac{\omega_2^{\beta+\frac{1}{2}}\omega_4^\beta}{\l\omega_2\r^{M/2}\l\omega_4\r^{M/2}}\,d\omega_4&\lesssim \omega_1^{\beta-M/2}\int_{2\omega_3-\omega_1}^\infty\frac{1}{\l x\r^{M/2-\beta-\frac{1}{2}}} \,dx\lesssim \omega_1^{\beta-M/2}.  
\end{align*}
Hence, we obtain
\begin{align}\label{UB C_22^234}
\mathcal{C}_{22}^{234}[n^{01},n^{01},n^{01}]&\lesssim \omega_1^{2\beta-\frac{1}{2}-M/2}\int_{\omega_1/2}^\infty \frac{1}{\l\omega_3\r^{M/2-\beta}}\,d\omega_3 \lesssim \omega_1^{3\beta+\frac{1}{2}-M}.
\end{align}
\item Upper bound for $\mathcal{C}_{3}^{234}[n^{01},n^{01},n^{01}]$: Given $\omega_1>1$, in the domain $\mathcal{D}_3$ we have $\omega_2\leq\omega_4$ and $\omega_4\geq\omega_1$, thus 
\begin{align*}
\int_{\omega_1}^\infty\frac{\omega_2^\beta\omega_4^\beta}{\l\omega_2\r^{M/2}\l\omega_4\r^{M/2}}\,d\omega_4 &\lesssim  \int_{\omega_1}^\infty\omega_4^{2\beta-M/2}\,d\omega_4\lesssim\omega_1^{2\beta+1-M/2}.
\end{align*}
Thus
\begin{align}\label{UB C_3^234}
\mathcal{C}^{234}_3[n^{01},n^{01},n^{01}]\lesssim  \omega_1^{2\beta+\frac{1}{2}-M/2}\int_0^{\omega_1} \frac{1}{\l\omega_3\r^{M/2-\beta-\frac{1}{2}}}\,d\omega_3 \lesssim \omega_1^{2\beta+\frac{1}{2}-M/2}.
\end{align}
\item Upper bound for $\mathcal{C}_{1}^{234}$: Given $\omega_1>1$, in the domain $\mathcal{D}_1$ we have $\omega_2\geq\omega_1$, so
\begin{align*}
\int_{\omega_3}^\infty\frac{\omega_2^\beta\omega_4^\beta}{\l\omega_2\r^{M/2}\l\omega_4\r^{M/2}}\,d\omega_4  \lesssim \omega_1^{\beta-M/2}\int_{\omega_3}^\infty\omega_4^{\beta-M/2}\,d\omega_4  \lesssim \omega_1^{\beta-M/2}.
\end{align*}
Therefore, since $\omega_3\geq\omega_1>1$
\begin{align}\label{UB C_1^234}
 \mathcal{C}^{234}_1[n^{01},n^{01},n^{01}]\lesssim  \omega_1^{2\beta-M/2}\int_{\omega_1}^{\infty} \omega_3^{\beta-M/2}\,d\omega_3 \lesssim \omega_1^{3\beta+1-M}.
\end{align}
\end{itemize}

Combining \eqref{lower bound C^234} with \eqref{UB C_21^234}-\eqref{UB C_1^234}, we conclude 
\begin{equation}\label{size of C^234}
\mathcal{C}^{234}[n^{01},n^{01},n^{01}]\approx\omega_1^{2\beta-\frac{1}{2}-M/2},\quad\omega_1>1.    
\end{equation}

\subsection{Conclusion}
With the estimates above we are ready to finish the proof of Theorem \ref{thm:gain-only}. 

\begin{proof}
Assume for contradiction that the gain-only equation \eqref{DKWE} is well-posed in $\l k_1 \r^{-M} L^{\infty}_{k_1}$ for $M>6.$ Consider the initial datum $f^{01}(k_1):= (1+\vert k_1 \vert^4)^{-M/4}$ and let $f(t,k_1)$ the corresponding unique solution.  

By Lemma \ref{isotropic-gain}, $n(t,\omega_1)=f(t,k_1)$ satisfies 
\begin{align}\label{gain equation integral form}
    n(t) = n^{01} + \int_0^t \left(\mathcal{C}^{234}[n,n,n] + \mathcal{C}^{134}[n,n,n]\right)(s) \, ds,\quad t\in[0,T],
\end{align}
with initial datum $n^{01}(\omega_1):= \langle \omega_1 \rangle^{-M}$.
Since $n^{01}\geq 0$ and $\mathcal{C}^{234},\mathcal{C}^{134}$ are monotone, a straightforward induction shows that successive Picard iterates are all larger than $n^{01}.$ Since they converge to the unique solution $n$ to \eqref{gain IVP}, we deduce that $n\geq n^{01}$. 

Thus  \eqref{gain equation integral form} and \eqref{lower bound C^234} imply
\begin{align*}
    n(t) \geq \int_0^t \left(\mathcal{C}^{234}[n^{01},n^{01},n^{01}] + \mathcal{C}^{134}[n^{01},n^{01},n^{01}]\right) ds \geq t \,\mathcal{C}^{234}[n^{01},n^{01},n^{01}] \gtrsim t \,\omega_1^{2\beta - \frac{1}{2} - M/2},
\end{align*}
for $\omega_1>1$.

Now multiplying both sides of the above inequality by $\langle \omega_1 \rangle^{M/2}$ and letting $\omega_1 \to +\infty,$ we obtain 
 $$\Vert \langle \omega_1 \rangle^{M/2} n(t) \Vert_{L^\infty_{\omega_1}([0,+\infty))} = +\infty,$$
 which is a contradiction.

 We conclude that gain-only equation \eqref{DKWE} is ill-posed in $\l k_1\r^{-M}L^\infty_{k_1}$.
\end{proof}

\section{Ill-posedness for \eqref{DKWE} if $\beta>\frac{1}{4}$} \label{illposed:full}
In this section we prove that the threshold of well-posedness for \eqref{DKWE} is the same as the gain equation. We stress that the mechanism for ill-posedness is different since cancellations between the gain and loss terms make the picture less clear. However,  we show that high frequency oscillations counteract for this effect.
\begin{theorem} \label{main-thm}
If $\beta > \frac{1}{4},$ then the equation \eqref{DKWE} is ill-posed in $\langle k_1 \rangle^{-M} L^\infty_{k_1}$ for $M>10.$
\end{theorem}

\begin{remark}
As will be clear from the proof, this statement is true regardless of the size of the initial datum. Put differently, well-posedness cannot be recovered by imposing a size condition on the initial datum.
\end{remark}

\begin{remark}
The proof also applies to the inhomogeneous equation (recall Remark \ref{rmk-inhomo}).
\end{remark}

\begin{remark}
We have not tried to optimize the condition $M>10$ since our main point is that well-posedness cannot be recovered by considering smoother initial datum.
\end{remark}

\subsection{Set-up}
To prove Theorem \ref{main-thm} we argue by contradiction: assume that \eqref{DKWE} is well-posed in the sense of \ref{definition of a solution}.

Let $A>0, N\in \mathbb{N}$ be large constants to be fixed later. Consider the initial datum
\begin{align*}
    f^{0}(k_1) := \frac{A + \cos(N \vert k_1 \vert^2)}{(1+\vert k_1 \vert^4)^{M/4}} \in \langle k_1 \rangle^{-M} L^\infty_{k_1}.
\end{align*}
Then by Lemma \ref{isotropic-full}, $n(t,\omega_1) =f(t,k_1)$ satisfies 
$$n(t)=n^{01}+\int_0^t \mathcal{C}[n](s)\,ds,\quad t\in[0,T],$$
with initial datum
\begin{align*}
n^0(\omega_1) := A n^{01}(\omega_1) + n^{02}(\omega_1)  , \ \ n^{01}(\omega_1) = 
\frac{1}{\langle \omega_1 \rangle^{M/2}}, \ \   n^{02}(\omega_1) = 
\frac{\cos(N \omega_1)}{\langle \omega_1 \rangle^{M/2}}.
\end{align*}

We isolate the main contribution in $\mathcal{C}$ making the following elementary manipulations:
\begin{align*}
& n^0_2 n^0_3 (n^0_4 - n^0_1)\nonumber  \\
=&  \big(A + \cos(N \omega_2) \big) \big(A + \cos(N \omega_3) \big) \bigg(  \frac{\cos(N \omega_4) - \cos(N \omega_1)}{\langle \omega_2 \rangle^{M/2} \langle \omega_3 \rangle^{M/2} \langle \omega_4 \rangle^{M/2}} + \frac{A (\langle \omega_1 \rangle^{M/2} - \langle \omega_4 \rangle^{M/2}) }{\langle \omega_1 \rangle^{M/2} \langle \omega_2 \rangle^{M/2} \langle \omega_3 \rangle^{M/2} \langle \omega_4 \rangle^{M/2}} \bigg)\nonumber \\
=& -A^2 \frac{\cos(N \omega_1)}{\langle \omega_2 \rangle^{M/2} \langle \omega_3 \rangle^{M/2} \langle \omega_4 \rangle^{M/2}} + A^2 \frac{\cos(N \omega_4)}{\langle \omega_2 \rangle^{M/2} \langle \omega_3 \rangle^{M/2} \langle \omega_4 \rangle^{M/2}} \nonumber   \\
& +  \big(A + \cos(N \omega_2) \big) \cos(N \omega_3)   \frac{\cos(N \omega_4) - \cos(N \omega_1)}{\langle \omega_2 \rangle^{M/2} \langle \omega_3 \rangle^{M/2} \langle \omega_4 \rangle^{M/2}}\\
&+   \big(A + \cos(N \omega_3) \big) \cos(N \omega_2)   \frac{\cos(N \omega_4) - \cos(N \omega_1)}{\langle \omega_2 \rangle^{M/2} \langle \omega_3 \rangle^{M/2} \langle \omega_4 \rangle^{M/2}} \nonumber  \\
&+ A \big(A + \cos(N \omega_2) \big) \big(A + \cos(N \omega_3) \big)  \frac{\langle \omega_1 \rangle^{M/2} - \langle \omega_4 \rangle^{M/2}}{\langle \omega_1 \rangle^{M/2} \langle \omega_2 \rangle^{M/2} \langle \omega_3 \rangle^{M/2} \langle \omega_4 \rangle^{M/2}}. 
\end{align*}
Next we add $n_1^0 n_4^0 (n_3^{0} - n_2^{0})$ on both sides of the above equation, multiply  by $S$ (defined in \eqref{iso-DKWE}) and integrate over $\{ (\omega_3,\omega_4) \in [0,\infty)^2\,: \omega_3 \leq \omega_4 \}$. Using the fact that $0\leq n^0\leq (A+1)n^{01}$, we obtain the lower bound
\begin{align*}
|\mathcal{C}[n^0,n^0,n^0]|&\geq A^2|\cos(N\omega_1)|\,\mathcal{C}^{234}[n^{01},n^{01},n^{01}]-A^2|\mathcal{C}^{234}[n^{01},n^{01},n^{02}]|-4(A+1)\mathcal{C}^{234}[n^{01},n^{01},n^{01}]\nonumber\\
&\quad-A(A+1)^2\mathcal{C}^{'234}[n^{01},n^{01},n^{01}]-(A+1)^3\mathcal{C}^{124}[n^{01},n^{01},n^{01}] - (A+1)^3 \mathcal{C}^{134}[n^{01},n^{01},n^{01}],
\end{align*}
where $\mathcal{C}^{'234}$ is defined as
\begin{equation}\label{nu definition}
\mathcal{C}^{'234}[n^{01},n^{01},n^{01}]=\omega_1^{\beta-\frac{1}{2}}\int_{0\le\omega_3\leq\omega_4,\,\omega_2\geq 0} S(\omega_1,\omega_2,\omega_3,\omega_4)\frac{\big|\l\omega_1\r^{M/2}-\l\omega_4\r^{M/2}\big|}{\l\omega_1\r^{M/2}\l\omega_2\r^{M/2}\l\omega_3\r^{M/2}\l\omega_4\r^{M/2}}\,d\omega_3\,d\omega_4,   
\end{equation}
and $S(\omega_1,\omega_2,\omega_3,\omega_4)$ is given in \eqref{iso collisional operator}. It corresponds to the operator $\mathcal{C}^{234}$ with the cross-section $S(\omega_1,\omega_2,\omega_3,\omega_4)\frac{|\l\omega_1\r^{M/2}-\l\omega_4\r^{M/2}|}{\l\omega_1\r^{M/2}}.$

Isolating the expected main contribution, we have 
\begin{align} \label{C-lowerbound}
\vert \mathcal{C}[n^0,n^0,n^0] \vert & \geq I_1 - I_2 - I_3 - I_4 - I_5 - I_6 ,
\end{align}
where 
\begin{align*}
I_1 & :=A^2  |\cos (N \omega_1)|\, \mathcal{C}^{234} [n^{01},n^{01},n^{01}], \\
I_2 & := A^2 \big( \vert \mathcal{C}^{234}_{21}[n^{01},n^{01},n^{02}] \vert + \vert \mathcal{C}^{234}_3[n^{01},n^{01},n^{02}]  \vert \big), \\
I_3 & := 4(A+1) \mathcal{C}^{234}[n^{01},n^{01},n^{01}], \\
 I_4&:= A (A+1)^2 \mathcal{C}'^{234}[n^{01},n^{01},n^{01}], \\
   I_5& := (A+1)^3 \left(\mathcal{C}^{124},[n^{01},n^{01},n^{01}]+\mathcal{C}^{134}[n^{01},n^{01},n^{01}]\right), \\
    I_6&:= A^2 \big( \mathcal{C}^{234}_{22}[n^{01},n^{01},n^{01}] +  \mathcal{C}^{234}_{1}[n^{01},n^{01},n^{01}] \big).
\end{align*}
We now bound all the terms above for $\omega_1>1$. 

\subsection{Estimating $I_1,I_3,I_6$} \label{I1I3}
By \eqref{size of C^234},  we have
\begin{align}
I_1 &\approx A^2 \vert \cos(N \omega_1) \vert \omega_1^{2\beta-\frac{1}{2}-M/2}\label{estimate I_1},\\
I_3 &\approx (A+1) \omega_1^{2\beta - M/2 - \frac{1}{2}}\label{estimate I_3}.
\end{align}
Moreover by \eqref{UB C_22^234}, \eqref{UB C_1^234} we also have
\begin{align}\label{estimate I_6}
    I_6 \lesssim A^2 \omega_1^{3\beta+1-M}.
\end{align}
The above implicit numerical constants are independent of $N,A.$

\subsection{Estimating $I_2$} \label{I2}
 In this case we perform one integration by parts in $\omega_4$ and then estimate from above. We find 

%keeping only one boundary term on the right-hand side
\begin{align*}
    \vert \mathcal{C}^{234}_{21}[n^{01},n^{01},n^{02}] \vert & \lesssim \frac{1}{N} \omega_1^{2\beta-\frac{1}{2}-M/2} \int_{0}^{\omega_1/2} \frac{\omega_3^{2\beta}}{\l\omega_3\r^{M/2}} d\omega_3 \\
    &+ \frac{1}{N} \omega_1^{\beta-\frac{1}{2}}  \int_{0}^{\omega_1/2} \frac{\omega_3^\beta}{\langle \omega_3 \rangle^{M/2}} \int_{\omega_1-\omega_3}^{\omega_1}  \bigg \vert\frac{\partial}{\partial\omega_4} \bigg( \frac{\omega_2^{\beta+1/2}}{\langle \omega_2 \rangle^{M/2}} \frac{\omega_4^\beta}{\langle \omega_4 \rangle^{M/2}} \bigg) \bigg \vert d\omega_4 d\omega_3.
\end{align*}
The boundary term is clearly bounded by $\frac{1}{N}\omega_1^{2\beta-\frac{1}{2}-M/2}$. For the other term, in the domain under $\mathcal{D}_{21}$ we have
\begin{align*}
    \bigg \vert \frac{\partial}{\partial \omega_4} \bigg( \frac{\omega_2^{\beta+1/2}}{\langle \omega_2 \rangle^{M/2}} \frac{\omega_4^\beta}{\langle \omega_4 \rangle^{M/2}} \bigg) \bigg \vert &\lesssim \frac{\omega_4^\beta}{\l\omega_4\r^{M/2} {\langle \omega_2 \rangle^{M/2}}}\left(\omega_2^{\beta-\frac{1}{2}}+\omega_2^{\beta+\frac{1}{2}}\l\omega_2\r^{-2}\right)\\
    &\hspace{2cm}+\frac{\omega_2^{\beta+\frac{1}{2}}}{\l\omega_2\r^{M/2} {\langle \omega_4 \rangle^{M/2}}}\left(\omega_4^{\beta-1}+\omega_4^{\beta}\l\omega_4\r^{-2}\right) \\
    & \lesssim \frac{\omega_4^\beta}{\l\omega_4\r^{M/2} {\langle \omega_2 \rangle^{M/2}}}\left(\omega_2^{\beta-\frac{1}{2}}+\omega_2^{\beta+\frac{1}{2}}\right),
\end{align*}
where we used $\omega_4\geq\omega_1-\omega_3\geq\omega_1/2>1/2$ for the last line.

We deduce the bound
\begin{align*}
    \int_{\omega_1-\omega_3}^{\omega_1}  \bigg \vert \frac{\partial}{\partial\omega_4} \bigg( \frac{\omega_2^{\beta+1/2}}{\langle \omega_2 \rangle^{M/2}} \frac{\omega_4^\beta}{\langle \omega_4 \rangle^{M/2}} \bigg) \bigg \vert d\omega_4 &\lesssim \omega_1^{\beta-M/2} \bigg( \int_{0}^{1} \omega_2^{\beta-\frac{1}{2}} d\omega_2 + \int_{1}^{\omega_3} \omega_2^{\beta+\frac{1}{2}-M/2} d\omega_2 \bigg) \\
    &\lesssim \omega_1^{\beta-M/2}.
\end{align*}
Therefore 
\begin{align*}
     \vert \mathcal{C}^{234}_{21}[n^{01},n^{01},n^{02}] \vert & \lesssim \frac{1}{N}\omega_1^{2\beta-\frac{1}{2}-M/2} + \frac{1}{N}\omega_1^{2\beta-\frac{1}{2}-M/2}\int_0^{\omega_1/2}\left(\frac{\omega_3^{2\beta}}{\l\omega_3\r^{M}}+\frac{\omega_3^\beta}{\l\omega_3\r^{M/2}}\right) \, d\omega_3 \\
     & \lesssim\frac{1}{N} \omega_1^{2\beta - M/2 -\frac{1}{2}},
\end{align*}
where the implicit (numerical) constant does not depend on $N$ or $A.$

A similar reasoning shows that the same estimate holds for $\vert \mathcal{C}^{234}_{3}[n^{01},n^{01},n^{02}] \vert.$

We conclude
\begin{align}\label{estimate I_2}
     I_2 \lesssim \frac{A^2}{N} \omega_1^{2\beta - M/2 -\frac{1}{2}}.
\end{align}

\subsection{Estimating $I_4$} \label{I4}
In the domain $\mathcal{D}_{21}$, the basic inequality
\begin{align*}
    \big \vert \langle \omega_4 \rangle^{M/2} - \langle \omega_1 \rangle^{M/2} \big \vert \lesssim \omega_1^{M/2-1} (\omega_1-\omega_4)=\omega_1^{M/2-1}(\omega_3-\omega_2),
\end{align*}
holds since $0 \leq \omega_4 \leq \omega_1$ and $\omega_1>1$.

Thus we can bound
\begin{align*}
    \mathcal{C}^{'234}_{21}[n^{01},n^{01},n^{01}] & \lesssim \omega_1^{\beta-3/2}   \int_{0}^{\omega_1/2} \frac{\omega_3^\beta}{\langle \omega_3 \rangle^{M/2}}  \int_{\omega_1 - \omega_3}^{\omega_1}  \frac{\omega_2^{\beta+1/2} \omega_4^\beta (\omega_3-\omega_2)}{\langle \omega_2 \rangle^{M/2} \langle \omega_4 \rangle^{M/2}} d\omega_4 d\omega_3.
\end{align*}
Similarly as for the estimate \eqref{UB C_21^234} above we use the fact that $\omega_2\leq\omega_3$ and $\omega_4\approx\omega_1$ in $\mathcal{D}_{21}$ to obtain
\begin{align*}
    \int_{\omega_1-\omega_3}^{\omega_1} \frac{\omega_2^{\beta+\frac{1}{2}} \omega_4^\beta(\omega_3 - \omega_2)}{\l \omega_2 \r^{M/2} \l \omega_4\r^{M/2}}  d\omega_4 \lesssim \omega_1^{\beta-M/2} \omega_3^{\beta+\frac{1}{2}} \int_{0}^{\omega_3} \frac{\omega_3-\omega_2}{\l \omega_2 \r^{M/2}} d\omega_2\lesssim \omega_1^{\beta-M/2}\omega_3^{\beta+\frac{3}{2}}.
\end{align*}
We conclude
\begin{align*}
     \mathcal{C}^{'234}_{21}[n^{01},n^{01},n^{01}] & \lesssim \omega_1^{2\beta - \frac{3}{2}-M/2} \int_{0}^{\omega_1/2} \frac{\omega_3^{2\beta+\frac{3}{2}}}{\l \omega_3 \r^{M/2}} d\omega_3 \lesssim \omega_1^{2\beta - \frac{3}{2}-M/2}.
\end{align*}

Next we estimate $\mathcal{C}^{234}_{3}$ similarly. On the domain $\mathcal{D}_3$ we have 
$$\big \vert \langle \omega_4 \rangle^{M/2} - \langle \omega_1 \rangle^{M/2} \big \vert \lesssim \omega_4^{M/2-1} \vert \omega_4 - \omega_1 \vert=\omega_4^{M/2-1}(\omega_2-\omega_3).$$ 
As a result
\begin{align*}
\mathcal{C}^{'234}_{3}[n^{01},n^{01},n^{01}] & \lesssim \omega_1^{\beta-\frac{1}{2}-M/2}   \int_{0}^{\omega_1} \frac{\omega_3^\beta}{\langle \omega_3 \rangle^{M/2}}  \int_{\omega_1}^{\infty}  \frac{\omega_2^{\beta+\frac{1}{2}} \omega_4^\beta (\omega_2-\omega_3)}{\langle \omega_2 \rangle^{M/2} \langle \omega_4 \rangle} \,d\omega_4\,d\omega_3\\
& \lesssim \omega_1^{2\beta-M/2-\frac{3}{2}}  \int_{0}^{\omega_1} \frac{\omega_3^\beta}{\langle \omega_3 \rangle^{M/2}} \int_{\omega_1}^{\infty}  \frac{\omega_2^{\beta+\frac{3}{2}} }{\langle \omega_2 \rangle^{M/2}} \,d\omega_4 \,d\omega_3 \\
& \lesssim \omega_1^{2\beta-M/2-\frac{3}{2}}  \int_{0}^{\omega_1} \frac{\omega_3^\beta}{\langle \omega_3 \rangle^{M/2}} \int_{\omega_3}^{\infty}  \frac{1}{\langle x \rangle^{M/2-\beta-\frac{3}{2}}} \,dx \,d\omega_3 \\
&\lesssim \omega_1^{2\beta-M/2-\frac{3}{2}} .
\end{align*}

A cruder bound as in \eqref{UB C_22^234}, \eqref{UB C_1^234} suffices   for $\mathcal{C}^{'234}_{22}$ and $\mathcal{C}^{'234}_{1}$ so we omit the proof. 

We conclude that
\begin{align} \label{size of C'234}
\vert \mathcal{C}^{'234}[n^{01},n^{01},n^{01}] \vert \lesssim  \omega_1^{2 \beta - 3/2-M/2},
\end{align}
and consequently
\begin{align} \label{estimate I_4}
    I_4 & \lesssim A^3 \omega_1^{2 \beta - 3/2-M/2}.
\end{align}

\subsection{Estimating $I_5$} As we will see this is a lower order term.  

\begin{itemize}
    \item Estimate for $\mathcal{C}^{124}_{21}[n^{01},n^{01},n^{01}]$, $\mathcal{C}^{134}_{21}[n^{01},n^{01},n^{01}]$: In $\mathcal{D}_{21}$ we have $\omega_4\approx\omega_1$ and $\omega_2\leq\omega_1$, thus
\begin{align*}
\int_{\omega_1-\omega_3}^{\omega_1}\frac{\omega_2^{\beta+\frac{1}{2}}\omega_4^\beta}{\l\omega_2\r^{M/2} \l\omega_4\r^{M/2}}\,d\omega_4 &\lesssim \omega_1^{2\beta+\frac{1}{2}-M/2}\int_{0}^\infty \frac{1}{\l x\r^{M/2}}\,dx \lesssim \omega_1^{2\beta+\frac{1}{2}-M/2}.   
\end{align*}
Hence, we obtain
\begin{align}\label{est C_21^124}
\mathcal{C}_{21}^{124}[n^{01},n^{01},n^{01}]\lesssim \omega_1^{3\beta-M}\int_0^{\omega_1/2}\omega_3^{\beta}\,d\omega_3\lesssim \omega_1^{4\beta+1-M}.    
\end{align}
In $\mathcal{D}_{21}$, we also have $\omega_2\leq\omega_3$, thus
\begin{align}\label{est C_21^134}
 \mathcal{C}_{21}^{134}[n^{01},n^{01},n^{01}]\lesssim \mathcal{C}_{21}^{124}[n^{01},n^{01},n^{01}]\lesssim \omega_1^{4\beta+1-M}.   
\end{align}
\item Estimate for $\mathcal{C}^{124}_{22}[n^{01},n^{01},n^{01}]$, $\mathcal{C}^{134}_{21}[n^{01},n^{01},n^{01}]$: In $\mathcal{D}_{22}$ we also have $\omega_4\approx\omega_1$ and $\omega_2\leq\omega_1$,  thus
\begin{align*}
\int_{\omega_3}^{\omega_1}\frac{\omega_2^{\beta+\frac{1}{2}}\omega_4^\beta}{\l\omega_2\r^{M/2}\l\omega_4\r^{M/2}}\,d\omega_4&\lesssim \omega_1^{2\beta+\frac{1}{2}-M/2}\int_{2\omega_3-\omega_1}^\infty\frac{1}{\l x\r^{M/2}}\,d x\lesssim \omega_1^{2\beta+\frac{1}{2}-M/2}.    
\end{align*}
Hence, we obtain
\begin{align}\label{est C_22^124}
 \mathcal{C}^{124}_{22}[n^{01},n^{01},n^{01}]&\leq \omega_1^{3\beta-M}\int_{\omega_1/2}^{\omega_1}\omega_3^\beta\,d\omega_3 \lesssim \omega_1^{4\beta+1-M}.  
\end{align}
In $\mathcal{D}_{22}$, we also have $\omega_2\leq\omega_3$, thus
\begin{align}\label{est C_22^134}
 \mathcal{C}_{22}^{134}[n^{01},n^{01},n^{01}]\lesssim \mathcal{C}_{22}^{124}[n^{01},n^{01},n^{01}]\lesssim \omega_1^{4\beta+1-M}.   
\end{align}

\item Estimate for $\mathcal{C}^{124}_{3}[n^{01},n^{01},n^{01}]$, $\mathcal{C}^{134}_{3}[n^{01},n^{01},n^{01}]$: In $\mathcal{D}_3$ we have $\omega_2\leq\omega_4$ and $\omega_4\geq\omega_1>1$, which imply
\begin{align*}
    \int_{\omega_1}^\infty \frac{\omega_2^\beta\omega_4^\beta}{\l\omega_4\r^{M/2}}\,d\omega_4\lesssim \int_{\omega_1}^\infty \omega_4^{2\beta-M/2}\,d\omega_4\approx \omega_1^{2\beta+1-M/2}. 
\end{align*}
Hence, we obtain 
\begin{align}\label{est C_3^134}
\mathcal{C}^{134}_3[n^{01},n^{01},n^{01}]\lesssim \omega_1^{3\beta+\frac{1}{2}-M}\int_{0}^{\omega_1}\frac{\omega_3^{\beta+\frac{1}{2}}}{\l\omega_3\r^{M/2}}\,d\omega_1\lesssim \omega_1^{3\beta+\frac{1}{2}-M}.    
\end{align}
In $\mathcal{D}_{3}$, we also have $\omega_3\leq\omega_2$, thus
\begin{align}\label{est C_3^134}
 \mathcal{C}_{3}^{124}[n^{01},n^{01},n^{01}]\lesssim \mathcal{C}_{3}^{134}[n^{01},n^{01},n^{01}]\lesssim \omega_1^{3\beta+\frac{1}{2}-M}.   
\end{align}
\item Estimate for $\mathcal{C}^{124}_{1}[n^{01},n^{01},n^{01}]$, $\mathcal{C}^{134}_{1}[n^{01},n^{01},n^{01}]$:
In $\mathcal{D}_{1}$, we have
\begin{align*}
\int_{\omega_3}^\infty \frac{\omega_2^\beta\omega_4^\beta}{\l\omega_4\r^{M/2}}\,d\omega_4\leq\int_{\omega_3}^\infty \frac{(\omega_4+\omega_3)^\beta\omega_4^\beta}{\l\omega_4\r^{M/2}}\,d\omega_4\   \lesssim \int_{\omega_3}^\infty\omega_4^{2\beta-M/2}\,d\omega_4\lesssim \omega_3^{2\beta+1-M/2}
\end{align*}
Hence
\begin{align}\label{est C_1^134}
\mathcal{C}_{1}^{134}[n^{01},n^{01},n^{01}]\lesssim \omega_1^{\beta-M/2}\int_{\omega_1}^\infty\omega_3^{2\beta+1-M/2}\,d\omega_3\lesssim \omega_1^{3\beta+2-M}.     
\end{align}
In $\mathcal{D}_1$ we also have $\omega_2\geq\omega_3$, thus
\begin{equation}\label{est C_1^124}
  \mathcal{C}_{1}^{124}[n^{01},n^{01},n^{01}]\lesssim  \mathcal{C}_{1}^{134}[n^{01},n^{01},n^{01}]  \lesssim \omega_1^{3\beta+2-M}
\end{equation}
\end{itemize}

Combining \eqref{est C_21^124}-\eqref{est C_1^124}, for $\omega_1>1$ we obtain
\begin{align} \label{size of C^124, C^134}
     \mathcal{C}_{1}^{124}[n^{01},n^{01},n^{01}],  \mathcal{C}_{1}^{134}[n^{01},n^{01},n^{01}] \lesssim \omega_1^{4\beta+2-M},
\end{align} 
and conclude
\begin{align}\label{estimate I_5}
 I_5\lesssim A^3 \omega_1^{4\beta+2-M}.  
\end{align}

\subsection{Conclusion of the proof}
We start by proving that
\begin{equation}\label{size of C^123}
\mathcal{C}^{123}[n^{01},n^{01},n^{01}]\lesssim\omega_1^{2\beta-\frac{1}{2}-M/2}.    
\end{equation}
Indeed we can write that 
\begin{align*}
n_1^{01} n_2^{01} n_3^{01} = n_2^{01} n_3^{01} n_4^{01} +  n_2^{01} n_3^{01} (n_1^{01}-n_4^{01}) = n_2^{01} n_3^{01} n_4^{01} + \frac{\langle \omega_4 \rangle^{M/2} - \langle \omega_1 \rangle^{M/2}}{\langle \omega_1 \rangle^{M/2} \langle \omega_2 \rangle^{M/2} \langle \omega_3 \rangle^{M/2} \langle \omega_4 \rangle^{M/2}} .
\end{align*}
After multiplying by $S$ and integrating over $\lbrace (\omega_3,\omega_4) \in [0;\infty)^2; \omega_3\leq \omega_4 \rbrace$ we find
\begin{align*}
    \mathcal{C}^{123}[n^{01},n^{01},n^{01}] \lesssim \mathcal{C}^{234}[n^{01},n^{01},n^{01}] + \vert \mathcal{C}^{'234}[n^{01},n^{01},n^{01}] \vert
\end{align*}
and \eqref{size of C^123} follows from \eqref{size of C^234} and \eqref{size of C'234}.

We conclude that by \eqref{size of C^123}, \eqref{size of C^234}, \eqref{size of C^124, C^134}, we have 
\begin{align} \label{C01-upperbound}
    \vert \mathcal{C}[n^{01},n^{01},n^{01}] \vert \lesssim \omega_1^{2\beta-\frac{1}{2}-M/2}. 
\end{align}

Next, let $C_i, i=1...6$ denote the numerical constants in the estimates of $I_i$ above \eqref{estimate I_1}, \eqref{estimate I_3}, \eqref{estimate I_6}, \eqref{estimate I_2}, \eqref{estimate I_4} and \eqref{estimate I_5}.  

Then by \eqref{C-lowerbound} we have 
\begin{align*}
  \big \vert  \mathcal{C}[n^0,n^0,n^0] \big \vert & \geq \omega_1^{2\beta-M/2-1/2} \bigg(A^2 C_1 \vert \cos(N \omega_1) \vert 
  - \frac{A^2 C_2}{N} 
  - (A+1) C_3 \\ 
  &- \frac{A^3}{\omega_1} C_4 - A^3 C_5  \omega_1^{2 \beta + \frac{5}{2} - M/2} - A^2 C_6 \omega_1^{\beta + \frac{3}{2}-M/2}  \bigg).
\end{align*}

Now we take $N$ such that $C_2/N < C_1/10,$ and $A$ such that $(A+1) C_3 < 1/10 A^2 C_1.$ Next write $\omega_1 = \frac{2\pi B}{N}$ and choose $B$ as an integer such that $\omega_1>10$ and $ \frac{N}{2 \pi B} (A^3 C_4 + A^3 C_5 + A^2 C_6) < \frac{A^2 C_1}{10}$, we can conclude (using that since $M>10$ and $\beta \in [0,1]$ we have $2\beta+5/2-M/2 < -1$)
\begin{align} \label{goodlowerbound}
  \langle \omega_1 \rangle^{M/2} \big \vert \mathcal{C}[n^0,n^0,n^0] \big \vert \geq \frac{A^2 C_1}{2} \omega_1^{2\beta-\frac{1}{2}}.
\end{align}

Next, note that by Definition \ref{definition of a solution} we have $n \in \mathcal{C}([0,T];\langle k \rangle^M L^\infty_k)$. Then for $0<\varepsilon<1$ to be fixed later, there exists $\eta>0$ such that for all $0 \leq t \leq \eta$, we have $\big \Vert \l \omega \r^{M/2} \big(n(t) -n^0 \big) \big \Vert_{L^\infty_\omega} < \varepsilon.$  

Then 
\begin{align*}
     n(t) -n^0  = \int_0^t \mathcal{C}[n,n,n] \, ds = t \mathcal{C}[n^0,n^0,n^0] + \int_0^t \mathcal{C}[n,n,n] - \mathcal{C}[n^0,n^0,n^0] \, ds.
\end{align*}
Then by multilinearity of $\mathcal{C}$ and \eqref{C01-upperbound}, we have 
\begin{align} 
\notag \big \vert \mathcal{C}[n,n,n] - \mathcal{C}[n^0,n^0,n^0] \big \vert & \leq \sum_{(f,g,h) \in \lbrace n-n^0, n^0 \rbrace, (f,g,h) \neq (n^0,n^0,n^0)} \vert \mathcal{C}[f,g,h] \vert \\
\notag & \lesssim \varepsilon (A+1)^3 \mathcal{C}[n^{01},n^{01},n^{01}] \\
\label{Cdiff} & \lesssim \varepsilon (A+1)^3 \omega_1^{2\beta-\frac{1}{2}-M/2} .
\end{align}
Hence, denoting $C_7$ the numerical constant in \eqref{Cdiff} above, we find for the parameters $A,B$ and $N$ defined above by \eqref{goodlowerbound} that 
\begin{align*}
    \vert n(t) \vert \geq t  \omega_1^{2\beta - M/2 - \frac{1}{2}} \big(\frac{A^2 C_1}{2} - C_7 \varepsilon(A+1)^3\big) \geq t \omega_1^{2\beta - M/2 - \frac{1}{2}} \frac{A^2 C_1}{4},
\end{align*}
where for the last inequality we took $\varepsilon$ small enough (e.g. smaller than $\frac{4 A^2 C_1}{(A+1)^3 C_7}$).

Now multiplying both sides of the above inequality by $\langle \omega_1 \rangle^{M/2}$ and letting $\omega_1 \to +\infty,$ we obtain 
 $$\Vert \langle \omega_1 \rangle^{M/2} n(t) \Vert_{L^\infty_{\omega_1}([0,+\infty))} = +\infty,$$
 which is a contradiction.

We conclude that \eqref{DKWE} is ill-posed in $\l k_1\r^{-M}L^\infty_{k_1}$.

\appendix

\section{Alternative proof of Theorem \ref{thm:gain-only}}
In this section we give a second proof of Theorem \ref{thm:gain-only} which does not rely on the isotropic formulation, rather closer in spirit to \cite{AmLe}.

\begin{proof}
Assume for contradiction that the equation is well posed in $\l k_1 \r^{-M} L^{\infty}_{k_1}.$ Let $f(t)$ denote the solution on $[0,T]$ with initial datum $f^0(k_1) = \langle k_1 \rangle^{-M}.$ Recall that $f(t,k_1) \geq f^0(k_1)$ for all $t \in [0,T].$  

Then write for $\vert k_1 \vert \geq 10$
\begin{align} \label{fTlowerbound}
\begin{split}
    f(T) & = f^0(k_1) + \int_0^T \mathcal{G}[f](t) \, dt \geq \int_0^T \mathcal{G}[f^0](t) \, dt \\
    & \geq \int_0^T \langle k_1 \rangle^{-M} \int_{\S^2} \int_{1 \leq \vert k_2 \vert \leq 2} \vert k_1-k_2 \vert \big( \vert k_1 \vert \vert k_2 \vert \vert k_1^* \vert \vert k_2^* \vert \big)^{2\beta} \langle k_2 \rangle^{-M} \langle k_2^* \rangle^{-M} \, d\sigma dk_2 dt \\
     &\gtrsim T \langle k_1 \rangle^{-M} \int_{1 \leq \vert k_2 \vert \leq 2}  \int_{\S^2} \vert k_1 \vert \big( \vert k_1 \vert \vert k_1^* \vert \big)^{2\beta} \langle k_2^* \rangle^{2\beta-M} \, d\sigma dk_2 .
\end{split}
\end{align}

Let
\begin{align*}
  I(k_1,k_2):=  \int_{\S^2} \frac{\vert k_1^* \vert^{2\beta}}{\langle k_2^* \rangle^{M-2\beta}} d\sigma.
\end{align*}
Denoting
\begin{align*}
    E:= \vert k_1 \vert^2 + \vert k_2 \vert^2 , \ \ \ V:= \frac{k_1+k_2}{2},\ \ \ u:= k_1-k_2  ,
\end{align*}
we find
\begin{align*}
    \vert k_1^* \vert^2 &= \frac{E}{2} - \vert u \vert \vert V \vert (\widehat{V} \cdot \sigma) \\
    \langle k_2^* \rangle^2 &= 1+  \frac{E}{2} + \vert u \vert \vert V \vert (\widehat{V} \cdot \sigma).
\end{align*}
Hence
\begin{align*}
    I(k_1,k_2) &= \int_{-1}^{1}  \frac{\big(\frac{E}{2} - \vert u \vert \vert V \vert z \big)^\beta}{\big(1 + \frac{E}{2} + \vert u \vert \vert V \vert z \big)^{M/2-\beta}} dz \geq \frac{1}{\vert u \vert \vert V \vert} \int_{-\vert u \vert \vert V \vert}^0 \frac{\big(\frac{E}{2} -  z \big)^\beta}{\big(1 + \frac{E}{2} +  z \big)^{M/2-\beta}} \, dz \\
    & \geq \frac{E^{\beta}}{2^\beta \vert u \vert \vert V \vert} \frac{1}{M/2-\beta-1}
    \bigg( \big( 1 + \frac{E}{2} - \vert u \vert \vert V \vert \big)^{\beta-M/2+1} - \big(1 + \frac{E}{2} \big)^{\beta-M/2+1} \bigg). 
\end{align*}
Next write that since $\vert k_2 \vert \leq 2,$
\begin{align*}
    \frac{E}{2} - \vert u \vert \vert V \vert = \frac{(k_1 \cdot k_2)^2}{\frac{E}{2} + \vert u \vert \vert V \vert} \leq 2 \big( \frac{k_1}{\vert k_1 \vert} \cdot k_2 \big)^2 \leq 8.
\end{align*}
As a result using $\beta - M/2+1 < 0,$ we find the lower bound 
\begin{align*}
    I(k_1,k_2) \gtrsim \frac{\vert k_1 \vert^{2\beta}}{\vert u \vert \vert V \vert} \bigg( 9^{\beta - M/2+1} - \big(1 + \frac{E}{2} \big)^{\beta-M/2+1} \bigg) .
\end{align*}
Taking $\vert k_1 \vert$ large enough, we find the lower bound 
\begin{align*}
     I(k_1,k_2) \gtrsim \vert k_1 \vert^{2 \beta-2}.
\end{align*}
Plugging this lower bound into \eqref{fTlowerbound} we find 
\begin{align*}
    f(T) \gtrsim  T \langle k_1 \rangle ^{-M} \cdot \vert k_1 \vert^{2\beta-2} \cdot \vert k_1 \vert^{1+2\beta}.
\end{align*}
Multiplying by $\langle k_1 \rangle^M$ and taking the limit $\vert k_1 \vert \to +\infty$ we find that for all $T>0,$
\begin{align*}
    \Vert \l k_1 \r^M f(T) \Vert_{L^{\infty}_{k_1}} = +\infty.
\end{align*}
We arrive at a contradiction.

\end{proof}

\color{black}


\begin{thebibliography}{1}

\bibitem{Am}
I. Ampatzoglou, \textit{Global well-posedness of the inhomogeneous kinetic wave equation near vacuum}, Kinet. Relat. Models,  Volume {\bf 17}, Issue 6: 838-854, (2024)


\bibitem{AmCoGer} 
I. Ampatzoglou, C. Collot, P. Germain, \textit{Derivation of the kinetic wave equation for quadratic dispersive problems in the inhomogeneous setting}, to appear in Amer. J. Math. (2024)

\bibitem{AmMiPaTa24} I. Ampatzoglou, J.K. Miller, N. Pavlovi\'c, M. Taskovi\'c, 	 \textit{Inhomogeneous wave kinetic equation and its hierarchy in polynomially weighted  $L^\infty$ spaces}, arXiv:2405.03984 (2024)

\bibitem{AmLe}
I. Ampatzoglou, T. L\'eger, \textit{Scattering theory for the inhomogeneous kinetic wave equation}, Preprint (2024)

\bibitem{BO}
A. B\'{e}nyi, T. Oh, \textit{The Sobolev inequality on the torus revisited}, Publ. Math. Debrecen \textbf{83} (2013), no. 3, 359--374
\bibitem{BGHS}
T. Buckmaster, P. Germain, Z. Hani, J. Shatah, \textit{Onset of the wave turbulence description of the longtime behavior of the nonlinear Schr\"{o}dinger equation}, Invent. Math. \textbf{225} (2021), no. 3, 787–855.
\bibitem{CH}
X. Chen, J. Holmer, \textit{Well/ill-posedness bifurcation for the Boltzmann equation with constant collision kernel}, to appear in Ann. PDE

\bibitem{CDG}
C. Collot, H. Dietert, P. Germain, \textit{Stability and cascades for the Kolmogorov-Zakharov spectrum of wave turbulence}, Arch. Ration. Mech. Anal. 248 (2024), \textbf{no. 1}, Paper No. 7, 31 pp

\bibitem{CG1}
C. Collot, P. Germain, \textit{On the derivation of the homogeneous kinetic wave equation}, to appear in Comm. Pure Appl. Math., arXiv:1912.10368


\bibitem{CG2}
C. Collot, P. Germain, \textit{Derivation of the homogeneous kinetic wave equation: longer time scales}, arXiv:2007.03508
\bibitem{DH1}
Y. Deng, Z. Hani, \textit{On the derivation of the wave kinetic equation for NLS}, Forum Math. Pi \textbf{9} (2021), Paper No. e6, 37 pp.
\bibitem{DH2}
Y. Deng, Z. Hani, \textit{Full derivation of the wave kinetic equation}, Invent. math. \textbf{233} (2023), no. 2, 543–724.
\bibitem{DH4}
Y. Deng, Z. Hani, \textit{Long time justification of wave turbulence theory}, arXiv:2311.10082
\bibitem{DHM}
Y. Deng, Z. Hani, X. Ma \textit{Long time derivation of Boltzmann equation from hard sphere dynamics}, Preprint arXiv:2408.07818 (2024)

\bibitem{EsMe24} M. Escobedo, A. Menegaki, \textit{Instability of singular equilibria of a wave kinetic equation}, 	arXiv:2406.05280

\bibitem{EV}
M. Escobedo, J.J.L. Vel\'{a}zquez, \textit{On the theory of weak turbulence for the nonlinear Schr\"{o}dinger equation}, Mem. Amer. Math. Soc. \textbf{238} (2015), no. 1124, v+107 pp.


\bibitem{Faou} E. Faou, \textit{Linearized Wave Turbulence Convergence Results for Three-Wave Systems}, Volume {\bf 378}, pages 807–849, (2020)


\bibitem{GeIoTr}
P. Germain, A. D. Ionescu, M.-B. Tran, \textit{Optimal local well-posedness theory for the kinetic wave equation}, J. Funct. Anal. \textbf{279} (2020), no. \textbf{4}, 108570, 28 pp.

\bibitem{GLZ}
P. Germain, J. La, K. Zhang, \textit{Local Well-posedness for the Kinetic MMT Model}, Preprint arXiv:2310.11893 (2023)

\bibitem{HaShZh} Z. Hani, J. Shatah, H. Zhu, \textit{Inhomogeneous turbulence for the Wick nonlinear Schrödinger equation}, Comm. Pure Appl. Math. {\bf 77}, 11, 4100-4142, (2024)


\bibitem{HaRoStTr} A. Hannani, M. Rosenweig, G. Staffilani, M.B. Tran, \textit{On the wave turbulence theory for a stochastic KdV type equation -- Generalization for the inhomogeneous kinetic limit}, Preprint: 	arXiv:2210.17445 (2022)



\bibitem{KPV1}
C. Kenig, G. Ponce, L. Vega, \textit{Small solutions to nonlinear Schr\"{o}dinger equations}, Ann. Inst. Henri Poincar\'{e} Analyse Nonlin\'{e}aire,  \textbf{10}, no. 3, (1993), 255--288

\bibitem{KPV2}
C. Kenig, G. Ponce, L. Vega, \textit{Smoothing effect and local existence theory for generalized nonlinear Schr\"{o}dinger equations}, Invent. math. \textbf{134}, (1998), 489--545

\bibitem{Ma} X. Ma, \textit{Almost sharp wave kinetic theory of multidimensional KdV type equations with $d\geq 3$}, Preprint: 	arXiv:2204.06148 (2022)



\bibitem{MMT}
A.J. Majda, D.W. McLaughlin, E.G. Tabak, \textit{A one-dimensional model for dispersive wave turbulence}, J. Nonlinear Sci. \textbf{6}, pp. 9-44 (1997)



\bibitem{me23} A. Menegaki, 
\text{$L^2$-stability near equilibrium for the 4 waves kinetic equation}, Kinet. Relat. Models, 2024, Volume 17, Issue 4: 514-532

\bibitem{Na} S. Nazarenko, \textit{Wave turbulence}, Lecture Notes in Physics, \textbf{825}. Springer, Heidelberg, 2011


\bibitem{ST}
G. Staffilani, M.-B. Tran, \textit{On the wave turbulence theory for a stochastic KdV type equation}, arXiv:2106.09819 (2021)

\bibitem{ZSKN}
Y. Zhu, B. Semisalov, G. Krstulovic, S. Nazarenko, \textit{Testing wave turbulence theory for the Gross-Pitaevskii system}, Physical Review E \textbf{106}, 014205 (2022)


\end{thebibliography}
\end{document}